\documentclass{amsart}

\usepackage{amssymb,amsmath,graphicx}

\usepackage{xcolor}

\usepackage{graphicx}
\usepackage{xypic}
\usepackage{capt-of}

\usepackage[shortlabels]{enumitem}

\usepackage[utf8]{inputenc}
\usepackage[T1]{fontenc}
\newtoks\prt

\numberwithin{equation}{section}

\newtheorem{thm}{Theorem}[section]

\newtheorem{question}[thm]{Question}

\newtheorem{lemma}[thm]{Lemma}

\newtheorem{prop}[thm]{Proposition}

\newtheorem{cor}[thm]{Corollary}

\newtheorem{theoreml}{Theorem}[section]

\theoremstyle{definition}

\newtheorem{remark}[thm]{Remark}

\newtheorem{definition}[thm]{Definition}

\def\eqn#1$$#2$${\begin{equation}\label#1#2\end{equation}}

\def\C{\mathcal C}

\def\ep{\varepsilon}

\def\en{\mathbb N}

\def\er{\mathbb R}

\def\spt{\operatorname{supp}}

\def \reg {\partial _{\kern1pt\text{reg}}}

\newcommand{\norm}[1]{\left\|#1\right\|}

\newcommand{\abs}[1]{\left| #1  \right|}

\title{On embedding separable spaces $\mathcal{C}(L)$ in arbitrary spaces $\mathcal{C}(K)$}
\author[J.~Rondo\v s]{Jakub Rondo\v s}
\address{Kurt G\"odel Research Center, Department of Mathematics, Vienna University, Vienna, Austria.}
\email{jakub.rondos@gmail.com}
\author[D.\ Sobota]{Damian Sobota}
\address{Kurt G\"odel Research Center, Department of Mathematics, Vienna University, Vienna, Austria.}
\email{damian.sobota@univie.ac.at}

\thanks{J. Rondo\v{s} was supported by the Austrian Science Fund (FWF):~I~5918-N. D. Sobota was supported by the Austrian Science Fund (FWF):~ESP~108-N}

\begin{document}

% Primary:
% 46E15 Banach spaces of continuous, differentiable or analytic functions
% 46B03 Isomorphic theory (including renorming) of Banach spaces
% 46B04 Isometric theory of Banach spaces
% Secondary:
% 47B38 Linear operators on function spaces (general)
% 54A25 Cardinality properties (cardinal functions and inequalities, discrete subsets)
% 54C35 Function spaces in general topology 


\begin{abstract}
%For a compact metric space $L$ and a compact space $K$, we provide several characterizations of the presence of either an isometric or an isomorphic copy of $\C(L)$ in $\C(K)$. We further describe the relative cellularity of the perfect kernel of $K$, and of its Cantor--Bendixson derived sets of countable order, in terms of the $\C(K)$ space. 
Supplementing and expanding classical results, for compact spaces $K$ and $L$, $L$ metric, and their Banach spaces $\mathcal{C}(L)$ and $\mathcal{C}(K)$ of continuous real-valued functions, we provide several characterizations of the existence of isometric, resp. isomorphic, embeddings of $\mathcal{C}(L)$ into $\mathcal{C}(K)$. In particular, we show that if the embedded space $\C(L)$ is separable, then the classical theorems of Holszty\'{n}ski and Gordon become equivalences.

We also obtain new results describing the relative cellularities of the perfect kernel of a given compact space $K$ and of the Cantor--Bendixson derived sets of $K$ of countable order in terms of the presence of isometric copies of specific spaces $\mathcal{C}(L)$ inside $\mathcal{C}(K)$. 
\end{abstract}

\maketitle

\section{Introduction\label{sec:intro}}

For a compact space $K$, by $\C(K)$ we denote as usual the Banach space of continuous real-valued functions on $K$ endowed with the supremum norm. Recall that $\C(K)$ is separable if and only if $K$ is metrizable. The main objective of this paper is to collect several classical theorems and combine them with our new results to obtain simple topological characterizations of the presence of isometric and isomorphic linear copies of $\C(L)$ inside  $\C(K)$, where $K, L$ are compact spaces and $L$ is metrizable, in particular in terms of Cantor--Bendixson derivatives of $K$ and $L$. For a compact space $K$ we also provide a new description of the relative cellularities of the Cantor--Bendixson derivatives $K^{(\alpha)}$ of countable order $\alpha$ and of the perfect kernel $K^{(\infty)}$ in terms of isometric embeddings of specific spaces $\C(L)$ into the space $\C(K)$.

Let us first briefly discuss the context of our study. The basic fact concerning isometric embeddings is that if $\rho \colon K \rightarrow L$ is a continuous surjection of compact spaces $K$ and $L$, then the composition operator $T \colon \C(L) \rightarrow \C(K)$, defined for every $f \in \C(L)$ by $T(f)=f\circ\rho$, is an isometric embedding which preserves lattice and algebraic operations. A partial converse to this is the classical result of Holszty\'{n}ski \cite{Holsztynski1966} which says that if $\C(L)$ is isometrically embedded in $\C(K)$, then there exists a closed subset of $K$ which maps continuously onto $L$. In general, the reverse implication does not hold---consider, for example, the spaces $L=[1, \omega_1]$ and $K=[0, 1]^{\omega_1}$, where $L$ is even homeomorphic to a subset of $K$, but $\C(L)$ cannot be isomorphically embedded into $\C(K)$, since $K$ satisfies the countable chain condition, while $L$ does not, cf. \cite[Theorem 1.3]{rondos_cardinal_invariants}.
%However, if a closed subset of a space $K$ maps continuously onto $L$,  it is possible to deduce the existence of a certain \emph{nonlinear} embedding of $\C(L)$ into $\C(K)$, see \cite[comments after Question 1.2]{Koszmider_reflection}.

Concerning isomorphic embeddings, the classical theorem of Gordon \cite{Gordon3} will be important to us---it asserts that if $K$ and $L$ are compact spaces and $T \colon \C(L) \rightarrow \C(K)$ is an isomorphic embedding with $\norm{T} \norm{T^{-1}}<3$ (in particular, if $T$ is an isometry), then for each ordinal number $\alpha$ we have $\abs{L^{(\alpha)}} \leq \abs{K^{(\alpha)}}$, where for a topological space $S$ and an ordinal number $\eta$ by $S^{(\eta)}$ we denote the $\eta$-th Cantor--Bendixson derivative of $S$ (see Section \ref{sec:prelim} for all relevant definitions). It is again easy to see that the reverse implication is not true in general---consider this time, for example, the spaces $L=[1,\omega_1]$ and $K=[0, 1]$ and note that $\C([1,\omega_1])$ is not separable. 

If one additionally assumes that $L$ is an uncountable metric compact space, then for any compact space $K$ the embeddability of $\C(L)$ into $\C(K)$ can be completely characterized using well-known results. Indeed, combining the celebrated theorems of Miljutin (see \cite[Theorem 2.1]{RosenthalC(K)}), Pe\l czy\'{n}ski and Semadeni \cite{Pelczynski_Semadeni_scattered}, Rosenthal \cite[Theorem 2.8]{RosenthalC(K)}, or above mentioned Holszty\'{n}ski \cite{Holsztynski1966} and Gordon \cite{Gordon3}, one can easily prove that in this setting the following assertions are equivalent:
%The question of embeddability of $\C(L)$ into $\C(K)$ has been completely solved in the case when $L$ is an uncountable compact metric space. It is known that in this setting the following assertions are equivalent (see e.g. Pe\l czy\'{n}ski and Semadeni \cite{Pelczynski_Semadeni_scattered} for the overview):
\begin{enumerate}[(a)]
    \item $\C(K)$ contains an isometric copy of $\C(L)$,
    \item $\C(K)$ contains an isomorphic copy of $\C(L)$,
    \item $\C(K)$ contains a positive isometric copy of $\C(L)$,
    \item there is an isomorphic embedding $T \colon \C(L) \rightarrow \C(K)$ with $\norm{T} \norm{T^{-1}} <3$,        
    \item there exists a closed subset of $K$ which maps continuously onto $L$,
    \item $\mathfrak{c}\le\abs{K^{(\alpha)}}$ for each ordinal number $\alpha$,
    \item $K$ is nonscattered.
\end{enumerate}
%Moreover, when $K$ is zero-dimensional, the above statements are also equivalent to the following two:
%
%\begin{enumerate}[(a)]
%\setcounter{enumi}{7}
%    \item $K$ maps continuously onto $L$,
%    \item there exists an isometric embedding of $\C(L)$ into $\C(K)$ which is also an embedding in the sense of Banach lattices and Banach algebras.
%\end{enumerate}

Similarly, if one initially assumes that $K$ is a nonscattered compact space, then for any metric compact space $L$ all the above statements easily hold true, again by the aforementioned classical results. However, it is quite surprising that until now there has not been known a characterization analogous to the one above which would work for all metrizable compact spaces $L$, regardless whether uncountable or not, and all compact spaces $K$, regardless whether nonscattered or not. We aim to present such a characterization in this work, but we will need to split it into two parts, separately for isometric embeddings (Theorem \ref{thm:mainA}) and for isomorphic embeddings (Theorem \ref{thm:mainB}).

Our first main result shows, among other things, that under the additional assumption that $L$ is metrizable both Holszty\'{n}ski's and Gordon's theorems become equivalences. For the definition of the height $ht(S)$ of a space $S$, see Section \ref{sec:prelim}.

\begin{theoreml}\label{thm:mainA}
For a compact metric space $L$ and a compact space $K$, the following assertions are equivalent:
\begin{enumerate}[(i)]
    \item $\C(K)$ contains an isometric copy of $\C(L)$,
    \item there exists a closed subset of $K$ which maps continuously onto $L$,
    \item $\abs{L^{(\alpha)}} \leq \abs{K^{(\alpha)}}$ for each ordinal number $\alpha$,
    \item $ht(L) \leq ht(K)$, and if $L$ is scattered, then 
    $\abs{L^{(ht(L)-1)}} \leq \abs{K^{(ht(L)-1)}}$,
    \item $\C(K)$ contains a positive isometric copy of $\C(L)$,
    \item there is an isomorphic embedding $T \colon \C(L) \rightarrow \C(K)$ with $\norm{T} \norm{T^{-1}} <3$.
\end{enumerate}
\end{theoreml}

Most of the implications in Theorem \ref{thm:mainA} are known; %what we noticed is that there is basically only one implication lacking in the state of the art to enclose the circle and hence to obtain full characterizations. Indeed, as the reader will surely notice, the place where the actual weight of our work lies is the demonstration 
what we found out is that there is primarily only one implication needed to be established for the full characterization circle to be closed. 
Indeed, as the reader will surely notice, in order to prove Theorem \ref{thm:mainA} we basically had only to show that if both spaces $K$ and $L$ are scattered and satisfy the inequality $\abs{L^{(ht(L)-1)}} \leq \abs{K^{(ht(L)-1)}}$, then $\C(L)$ embeds positively isometrically into $\C(K)$, that is, that implication (iv)$\Rightarrow$(v) holds for such $K$ and $L$. This is essentially the place where the actual weight of our work lies. It should be however noted that all our methods are general and work regardless whether $K$ is scattered or not. It is just a coincidental case that, as we said earlier, when $K$ is nonscattered, the required arguments can be also easily derived from a simple combination of well-known and classical results (see Remark \ref{remark:first}.(1) for details).
%In spite of the above-mentioned fact that most of the proofs needed for the above result have been done decades ago, we believe that due to its simple formulation and general nature is worth to state it and perform the work needed to finish its proof in this paper. 

For the proof of Theorem \ref{thm:mainA} as well as for additional equivalent conditions (vii)--(x), in particular in terms of spaces $\C(K,E)$ or for the special case of $K$ zero-dimensional, see Theorems \ref{main-isometric} and \ref{main-isometric-zero}. 

Let us point out that implication (i)$\Rightarrow$(v) of Theorem \ref{thm:mainA} deserves special attention, since, as far as we know, the question whether the existence of an isometric copy of an arbitrary space $\C(L)$ in another space $\C(K)$ implies the existence of a positive isometric copy of $\C(L)$ in $\C(K)$ is open, see Question \ref{question}.

To characterize the existence of isomorphic embeddings, we will utilize the notion of the Szlenk index of a Banach space $E$, denoted as $Sz(E)$. For the definition of the index and its basic properties see e.g. Lancien \cite{Lancien2006}. Note that the Szlenk index is a monotonic function, invariant under isomorphic embeddings. On the other hand, the inequality $Sz(\C(L))\le Sz(\C(K))$ does not imply the existence of an isomorphic copy of $\C(L)$ in $\C(K)$ (consider here, for example, again the Banach spaces $\C([1, \omega_1])$ and $\C([0, 1])$). Recall also the remarkable result by Causey \cite{CAUSEY_C(K)_index} (cf. also \cite{Causey_szlenk_hulls}) stating that the Szlenk index $Sz(\C(K))$ of a given space $\C(K)$ can be directly computed from the height $ht(K)$ of $K$ (see Section \ref{sec:main} for details). 

Using the Szlenk index, the classical isomorphic classification of separable $\C(K)$-spaces due to Bessaga and Pe\l czy\'{n}ski \cite{BessagaPelcynski_classification} and Miljutin \cite[Theorem 2.1]{RosenthalC(K)} can be formulated elegantly and simply: two separable $\C(K)$-spaces are isomorphic if and only if they have the same Szlenk index. In our second main theorem, obtained by combining known results and auxiliary propositions first used to prove Theorem \ref{thm:mainA}, we show that the existence of isomorphic embeddings of separable spaces $\C(L)$ in arbitrary spaces  $\C(K)$ can be characterized similarly. For the proof of Theorem \ref{thm:mainB} as well as for its extension in terms of spaces $\C(K,E)$, see Theorem \ref{main-isomorphic}.

\begin{theoreml}
\label{thm:mainB}
For a compact metric space $L$ and a compact space $K$, $\C(K)$ contains an isomorphic copy of $\C(L)$ if and only if $Sz(\C(L)) \leq Sz(\C(K))$.
\end{theoreml}

Our remaining main results are devoted to the novel descriptions of the relative cellularity $c(K^{(\alpha)},K)$ of a derived set $K^{(\alpha)}$ of a compact space $K$ in terms of the Banach space $\C(K)$ (see, again, Section \ref{sec:prelim} for the definitions). For basic properties of the relative cellularity see Pasynkov \cite{PASYNKOV_relative_cellularity} and Arhangel'skii \cite{Arhangelskii_cmuc}, and for its connections to isomorphic embeddings of $\C(K)$-spaces see Rondo\v{s} \cite{rondos_cardinal_invariants}. Let us here recall that, by Rosenthal \cite[Theorem on page 230]{Rosenthal_L^infty}, the cellularity $c(K)$ of $K$ is equal to the supremum of cardinal numbers $\kappa$ such that $\C(K)$ contains an isomorphic copy of the space $c_0(\kappa)$, where, for a set $\Gamma$, $c_0(\Gamma)$ stands for the Banach space of all indexed collections $\{a_{\gamma}\}_{\gamma \in \Gamma}$ of real numbers such that for each $\ep>0$ the set $\{\gamma \in \Gamma \colon \abs{a_{\gamma}}>\ep \}$ is finite, equipped with the supremum norm. We provide similar descriptions of the numbers 
$c(K^{(\alpha)}, K)$, where $\alpha$ is a countable ordinal, and $c(K^{(\infty)}, K)$ in terms of isometric embeddings. Here, for a set $\Gamma$ and a compact space $L$, $K_{\Gamma,L}$ stands for the one-point compactification of the disjoint union of $|\Gamma|$ many copies of $L$.

\begin{theoreml}\label{thm:mainC}
For an infinite compact space $K$ and a countable ordinal $\alpha$, we have
\[c(K^{(\alpha)}, K)=\sup  \Big\{ \abs{\Gamma} \colon \C\big(K_{\Gamma, [1, \omega^{\alpha}]}\big) \text{ embeds isometrically into } \C(K)\Big\}.\]
\end{theoreml}

\begin{theoreml}\label{thm:mainD}
For an infinite compact space $K$, we have
\[c(K^{(\infty)}, K)=\sup  \Big\{ \abs{\Gamma} \colon \C\big(K_{\Gamma, [0, 1]}\big) \text{ embeds  isometrically into } \C(K)\Big\}.\]
\end{theoreml}

For the proofs of extended versions of the above theorems, see Theorems \ref{Cellularity of derived sets} and \ref{Cellularity of perfect kernel}. Note that Theorem \ref{thm:mainD} can be understood as a quantitative generalization of the classical theorem of Pe\l czy\'{n}ski and Semadeni \cite{Pelczynski_Semadeni_scattered} saying that a compact space $K$ is nonscattered (i.e. $K^{(\infty)}\neq\emptyset$) if and only if $\C([0, 1])$ embeds isometrically into $\C(K)$ (see Section \ref{sec:prelim}).%, and, if $K$ is zero-dimensional, this is equivalent to saying that $K$ maps continuously onto the Cantor space.

Let us mention that our study is related to the area of research which investigates operators fixing copies of separable $\C(K)$-spaces. We recall that for a Banach space $E$ and compact spaces $K$, $L$, with $L$ metrizable, an operator $T \colon \C(K) \rightarrow E$ is said to \emph{fix an isometric}, resp. \emph{isomorphic}, \emph{copy} of $\C(L)$, if there exists a closed linear subspace $X$ of $\C(K)$ which is isometric, resp. isomorphic, to $\C(L)$ and such that $T$ restricted to $X$ is an isomorphism onto the image $T[X]$. Of course, if $T$ fixes a copy of $\C(L)$, then, in particular, $\C(L)$ is embedded into $\C(K)$. This area of research turned out to be very challenging and obtained results are often quite technical and intricate. They also heavily rely on the space $\C(L)$, e.g., surprisingly, the behavior of the space $\C([1, \omega^{\omega^2}])$ in this context is completely different from the one of the spaces $c_0$ and $\C([1 , \omega^{\omega}])$. For more information on the topic, see e.g. \cite{Alspach_operators_on_separable}, \cite{Bourgain1979}, \cite{Gasparis_problem_of_Rosenthal}, \cite{RosenthalC(K)}, or \cite{Wolfe_fixing_operators}.
Our main results show that when we forget about operators and only care about copies of $\C(L)$ in $\C(K)$, the situation becomes much simpler and the results can be formulated elegantly using natural topological properties of the considered compact spaces.

The paper is organized as follows. In the following preparatory section, Section \ref{sec:prelim}, we recall some standard notations and topological notions. In Section \ref{sec:aux} we prove several auxiliary lemmas. The proofs of Theorems \ref{thm:mainA} and \ref{thm:mainB} are contained in Section \ref{sec:main}. Finally, Section \ref{sec:cell} is devoted to the cellularity of Cantor--Bendixson derived sets of compact spaces, it contains in particular the proofs of Theorems \ref{thm:mainC} and \ref{thm:mainD}.

\section{Preliminaries \label{sec:prelim}}

For a set $X$ by $\abs{X}$ we denote its cardinality. $\en$ denotes the set of all positive integers, i.e. $\en=\{1,2,3,\ldots\}$. By $\omega$ we denote the first countable limit ordinal number; note that $\omega=\{0\}\cup\en$. $\omega_1$ denotes the first uncountable ordinal number and $\mathfrak{c}$ denotes the cardinality of the real line $\er$.

Throughout the paper, all compact spaces are assumed to be Hausdorff, and all Banach spaces are tacitly assumed to be real and of dimension at least one.  If $K$ is a compact space and $E$ is a Banach space, then $\C(K, E)$ stands for the Banach space of continuous $E$-valued functions on $K$ endowed with the supremum norm. Further, $\C(K)$ stands as usual for $\C(K, \mathbb{R})$. $\C(K)$ is of course also a Banach lattice and a Banach algebra with the natural pointwise operations. We set $\C(K,[0,1])=\{f\in\C(K)\colon 0\le f(x)\le 1\text{ for every }x\in K\}$. 

For each function $f\in\C(K)$, we define the \emph{support} $\spt(f)$ of $f$ as usual: $\spt(f)=\overline{\{x\in K\colon f(x)\neq0\}}$. Also, for a subset $U\subseteq K$, we say that $f$ is \emph{supported} by $U$ if $\spt(f)\subseteq U$. For each $x \in K$ we standardly set $f^+(x)=\max \{f(x), 0\}$ and $f^-(x)=\max \{-f(x), 0\}$ (so $f^-(x)\ge0$). Of course, $f^+,f^-\in\C(K)$ and $\|f\|=\max\{\|f^+\|,\|f^-\|\}$. If $A\subseteq K$, then $\chi_A$ denotes the characteristic function of $A$ in $K$.

Given Banach spaces $E$ and $F$, if $E$ and $F$ are isomorphic (resp. isometric), then we write $E\simeq F$ (resp. $E\cong F$). If there is an isomorphic embedding (resp. a (positive/PNPP\footnote{The meaning of the abbrevation \emph{PNPP} will be explained in the next section.}) isometric embedding) $T\colon E\to F$, then the image $T[E]$ is called an \emph{isomorphic copy} of $E$ in $F$ (resp. a \emph{(positive/PNPP) isometric copy}).

Cantor--Bendixson derivatives of topological spaces will constitute the main tool in our investigations. Let $S$ be a topological space. Set $S^{(0)}=S$ and  let $S^{(1)}$ be the set of all accumulation points of $S$ or, equivalently, $S^{(1)}=S\setminus\{x\in S\colon x\text{ is isolated in }S\}$. Further, for an ordinal number $\alpha>1$, let $S^{(\alpha)}=(S^{(\beta)})^{(1)}$ if $\alpha=\beta+1$, and $S^{(\alpha)}=\bigcap_{\beta<\alpha} S^{(\beta)}$ if $\alpha$ is a limit ordinal. For each ordinal $\alpha$ the set $S^{(\alpha)}$ is called the \emph{Cantor--Bendixson derivative} or the \emph{Cantor--Bendixson derived set} of $S$ of order $\alpha$. Note that for any two open sets $U\subseteq V\subseteq S$ and an ordinal $\alpha$ we have $U\cap V^{(\alpha)}\subseteq U^{(\alpha)}$, see e.g \cite[Lemma 2.3.(b)]{rondos-somaglia}.

For $S$, there exists an ordinal $\alpha$ such that for each $\beta>\alpha$ it holds $S^{(\beta)}=S^{(\alpha)}$. For such an ordinal $\alpha$, we put $S^{(\infty)}=S^{(\alpha)}$; the set $S^{(\infty)}$ is called the \emph{perfect kernel} of $S$. $S$ is said to be \emph{scattered} if $S^{(\infty)}=\emptyset$---in this case the minimal ordinal $\alpha$ such that $S^{(\alpha)}=\emptyset$ is called the \emph{height} of $S$ and is denoted by $ht(S)$. If $S$ is not scattered, then let $ht(S)=\infty$; we use the convention that $\alpha<\infty$ for each ordinal $\alpha$. It is easy to see that if $S$ is a scattered compact space, then $ht(S)$ is a successor ordinal and $S^{(ht(S)-1)}$ is a finite set. Further, it is well-known that a compact metric space is scattered if and only if it is countable. We also recall that for each ordinal number $\alpha$ and the ordinal interval $[1,\omega^{\alpha}]$, we have $([1, \omega^{\alpha}])^{(\alpha)}=\{\omega^{\alpha}\}$, which can be easily proved by transfinite induction, and thus $ht([1, \omega^{\alpha}])=\alpha+1$.

A family $\mathcal{U}$ of nonempty open subsets of $S$ is \emph{cellular} if $U\cap V=\emptyset$ for every distinct $U,V\in\mathcal{U}$. For a subset $A\subseteq S$, the \emph{relative cellularity} $c(A, S)$ is defined to be the supremum of cardinalities of cellular families $\mathcal{U}$ of $S$ such that $A\cap U\neq\emptyset$ for every $U\in\mathcal{U}$. Of course, for $A=S$, $c(A, S)$ coincides with the standard notion of the \emph{cellularity} $c(S)$ of $S$.

Recall that $S$ is said to be \emph{zero-dimensional} if $S$ admits a base consisting of clopen sets. 

We will frequently use the following characterization of nonscattered compact spaces due to Pe\l czy\'{n}ski and Semadeni \cite[Main Theorem]{Pelczynski_Semadeni_scattered}: a compact space $K$ is nonscattered if and only if $K$ maps continuously onto the unit interval $[0,1]$ if and only if $\C([0,1])$ isometrically embeds into $\C(K)$. Moreover, if $K$ is zero-dimensional, then $K$ is nonscattered if and only if $K$ maps continuously onto the standard Cantor space $\Delta$.

We will also need the classical Mazurkiewicz--Sierpi\'{n}ski classification of countable compact spaces: each countable compact space $K$ is homeomorphic to the ordinal interval $[1, \omega^{\alpha}m]$, where $\alpha=ht(K)-1$ and $m=\abs{K^{(\alpha)}}$ (see \cite[Theorem 8.6.10]{semadeni}).

\section{Auxiliary results\label{sec:aux}}

In this section we prove a number of auxiliary lemmas needed for the proofs of our main results presented in Section \ref{sec:main}. We first discuss some facts concerning the height and relative cellularity of subsets of general compact spaces. In the second subsection we provide several results concerning decompositions of continuous functions as well as pertaining to one-point compactifications of topological sums of families of compact spaces and their spaces of continuous functions. In the last subsection we show how the obtained statements apply to the case of countable compact spaces.

\subsection{Height and relative cellularity of compact spaces}

In the first simple lemma we show how the countable Cantor--Bendixson height of an open subset $U$ of a compact space $K$ naturally corresponds to the presence of a certain tree having open subsets of $U$ as nodes. We will not explicitly work with these families of open sets as with trees, but the tree structure is naturally implicitly present, with the ordering given by inclusion and incomparable elements of the tree corresponding to open sets with empty intersection. For families of clopen sets such trees have been considered e.g. by Bourgain, see \cite[page 36]{RosenthalC(K)}.   

\begin{lemma}
\label{Split}
Let $K$ be a compact space, $U$ be an open subset of $K$, and $\alpha$ be an ordinal such that $ht(U) >\alpha$. 
\begin{enumerate}[(i)]
    \item If $\alpha=\beta+1$ is a successor ordinal, then there exists a sequence $(U_n)_{n \in \en}$ of nonempty open sets with pairwise disjoint closures and such that $\overline{U_n}\subseteq U$ and $ht(U_n) \geq \alpha=\beta+1$ for each $n \in \en$.   
    \item If $\alpha$ is a limit ordinal and $(\beta_n)_{n \in \en}$ is an increasing sequence of ordinals converging to $\alpha$, then there exists a sequence $(U_n)_{n \in \en}$ of nonempty open sets with pairwise disjoint closures and such that $\overline{U_n}\subseteq U$ and $ht(U_n) \geq \beta_n+1$ for each $n \in \en$.   
\end{enumerate}
Moreover, if $K$ is zero-dimensional, then both in (i) and (ii) we may require that the sets $U_n$ are clopen.
\end{lemma}

\begin{proof}
(i) By the assumption, the set $U^{(\alpha)}$ is nonempty, and hence the set $U^{(\beta)}$ is infinite. Since $U^{(\beta)}$ is contained in the open set $U$, by a standard argument we get that there exist infinitely many points $(x_n)_{n \in \en}$ in $U^{(\beta)}$ and open subsets $(U_n)_{n \in \en}$ of $K$ with pairwise disjoint closures and such that $x_n\in U_n\subseteq \overline{U_n}\subseteq U$ for each $n\in\en$. Further, for each $n \in \en$, since $x_n \in U^{(\beta)}\cap U_n\subseteq U_n^{(\beta)}$, it follows that $ht(U_n) \geq \alpha=\beta+1$.

(ii) Let us fix a point $y \in U^{(\alpha)}$. Next, we pick a point $x_1\in U^{(\beta_1)}$, distinct from $y$. Since $U$ is open, we can find an open neighborhood $U_1$ of $x_1$ such that $\overline{U_1} \subseteq U$ and $y\not\in\overline{U_1}$. We have $x_1\in U^{(\beta_1)}\cap U_1\subseteq U_1^{(\beta_1)}$, so $ht(U_1)\ge\beta_1+1$. 
Since $y\in U^{(\alpha)}\cap (U\setminus\overline{U_1})\subseteq(U\setminus\overline{U_1})^{(\alpha)}$, we have $ht(U \setminus \overline{U_1}) \geq \alpha+1$, hence we can find $x_2 \in (U \setminus \overline{U_1})^{(\beta_2)}$, distinct from $y$. We further find an open neighborhood $U_2$ of $x_2$ such that $\overline{U_2} \subseteq (U \setminus \overline{U_1})$ and $y\not\in\overline{U_2}$. We have $x_2\in(U\setminus\overline{U_1})^{(\beta_2)}\cap U_2\subseteq U_2^{(\beta_2)}$, so $ht(U_2)\ge\beta_2+1$. We continue this inductive process until we find the desired sequence $(U_n)_{n \in \en}$ of open sets.

If $K$ is zero-dimensional, then both in (i) and (ii) for each point $x_n$ find a basic clopen neighborhood $U_n$.
\end{proof}

In what follows, we will explore how the relative cellularity of derived sets of a compact space $K$ behaves with respect to continuous surjections of closed subsets of $K$ and with respect to isometric embeddings of the space $\C(K)$.

\begin{lemma}
\label{lem_celularity}
Assume that $K$, $L$ are compact spaces such that there are a closed subset $F$ of $K$ and a continuous surjection $\rho\colon F\to L$. Then, for every open subset $U$ of $L$ and ordinal $\alpha$, if $U^{(\alpha)}\neq\emptyset$, then $(\rho^{-1}[\overline{U}])^{(\alpha)}\neq\emptyset$. 
\end{lemma}
\begin{proof}
Fix an open subset $U$ of $L$. We will proceed by induction on the ordinal $\alpha$. The case $\alpha=0$ is clear, so assume that $\alpha>0$ and that the statement is true for each $\beta<\alpha$. Suppose that $U^{(\alpha)}\neq\emptyset$. 

If $\alpha=\beta+1$ is a successor ordinal, we use Lemma \ref{Split}.(i) to find nonempty open subsets $(U_n)_{n \in \en}$ of $L$ with pairwise disjoint closures and such that $\overline{U_n}\subseteq U$ and $U_n^{(\beta)}\neq\emptyset$ for each $n \in \en$. By the inductive assumption, $(\rho^{-1}[\overline{U_n}])^{(\beta)} \neq \emptyset$ for each $n \in \en$. Further, since the sets $\big((\rho^{-1}[\overline{U_n}])^{(\beta)}\big)_{n\in\en}$ are pairwise disjoint subsets of $\rho^{-1}[\overline{U}]$, it follows that the set $(\rho^{-1}[\overline{U}])^{(\beta)}$ contains infinitely many distinct points. Thus, since the set $\rho^{-1}[\overline{U}]$ is compact, we have $(\rho^{-1}[\overline{U}])^{(\alpha)}=(\rho^{-1}[\overline{U}])^{(\beta+1)}\neq\emptyset$.

The step when $\alpha$ is a limit ordinal is trivial. Indeed, by the inductive assumption, for each $\beta<\alpha$, we have $(\rho^{-1}[\overline{U}])^{(\beta)}\neq\emptyset$. Thus, since the collection $\{(\rho^{-1}[\overline{U}])^{(\beta)} \colon \beta<\alpha \}$ is a decreasing family of nonempty compact sets, it has nonempty intersection.
\end{proof}

\begin{prop}
\label{celularity}
Assume that $K$, $L$ are compact spaces such that there is a closed subset $F$ of $K$ which maps continuously onto $L$. Then, for each ordinal $\alpha$, we have
\[c(L^{(\alpha)}, L) \leq c(F^{(\alpha)}, F).\]
%In particular, $ht(L) \leq ht(K)$, and if $L$ is scattered, then $\abs{L^{(ht(L)-1)}} \leq \abs{K^{(ht(L)-1)}}$.
\end{prop}

\begin{proof}
Let $\rho \colon F \rightarrow L$ be a continuous surjection. Let $\alpha$ be an ordinal number, and let $\mathcal{U}$ be a cellular family in $L$ such that each of its members intersects the derived set $L^{(\alpha)}$. For each $U \in \mathcal{U}$, we find an open set $V_U$ in $L$ such that $\overline{V_U} \subseteq U$ and $V_U\cap L^{(\alpha)}\neq\emptyset$. %For each $U \in \mathcal{U}$ we have $V_U^{(\alpha)}\neq\emptyset$ and $(\rho^{-1}[\overline{V_U}])^{(\alpha)}\subseteq(\rho^{-1}[U])^{(\alpha)} $, so Lemma \ref{lem_celularity} implies that $(\rho^{-1}[U])^{(\alpha)}\neq\emptyset$. Thus, since $\{ \rho^{-1}[U] \colon U \in \mathcal{U} \}$ is a cellular family in $F$ and $F^{(\alpha)}\cap(\rho^{-1}[U])=(\rho^{-1}[U])^{(\alpha)}\neq\emptyset$ for every $U\in\mathcal{U}$, we get the required inequality $c(L^{(\alpha)}, L) \leq c(F^{(\alpha)}, F)$. 
For each $U \in \mathcal{U}$ we have $V_U^{(\alpha)}\neq\emptyset$ and the set $\rho^{-1}[\overline{V_U}]$ is closed, so, by Lemma \ref{lem_celularity}, we have
\[\emptyset\neq(\rho^{-1}[\overline{V_U}])^{(\alpha)}\subseteq\rho^{-1}[\overline{V_U}]\cap F^{(\alpha)}\subseteq\rho^{-1}[U]\cap F^{(\alpha)}.\]
Thus, since $\{ \rho^{-1}[U] \colon U \in \mathcal{U} \}$ is a cellular family in $F$, we get the required inequality $c(L^{(\alpha)}, L) \leq c(F^{(\alpha)}, F)$. 
%For the second part of the conclusion, by the above argument, we immediately get that $ht(L) \leq ht(F) \leq ht(K)$. Moreover, if $L$ is scattered, then the set $L^{(ht(L)-1)}$ is finite, and it is therefore clear that
%\[ \abs{L^{(ht(L)-1)}}=c(L^{(ht(L)-1)}, L) \leq c(F^{(ht(L)-1)}, F) \leq \abs{F^{(ht(L)-1)}} \leq \abs{K^{(ht(L)-1)}}.\]
\end{proof}

\begin{cor}\label{celularity_cor}
    Assume that $K$, $L$ are compact spaces such that there is a closed subset of $K$ which maps continuously onto $L$. Then, $ht(L) \leq ht(K)$, and if $L$ is scattered, then $\abs{L^{(ht(L)-1)}} \leq \abs{K^{(ht(L)-1)}}$.
\end{cor}
\begin{proof}
Let $F$ be a closed subset of $K$ which maps continuously onto $L$. By Proposition \ref{celularity} we immediately get that $ht(L) \leq ht(F) \leq ht(K)$. Moreover, if $L$ is scattered, then the set $L^{(ht(L)-1)}$ is finite, and it is therefore clear that
\[ \abs{L^{(ht(L)-1)}}=c(L^{(ht(L)-1)}, L) \leq c(F^{(ht(L)-1)}, F) \leq \abs{F^{(ht(L)-1)}} \leq \abs{K^{(ht(L)-1)}}.\]
\end{proof}

\begin{prop}
\label{isometry_and_cellularity}
Let $K$, $L$ be compact spaces such that $\C(L)$ is isometrically embedded into $\C(K)$. Then, for each ordinal $\alpha$, $c(L^{(\alpha)}, L) \leq c(K^{(\alpha)}, K)$.
\end{prop}

\begin{proof}
Let $T \colon \C(L) \rightarrow \C(K)$ be an isometric embedding. For each $x \in L$  set 
%\[\Lambda_x=\Big\{y \in K \colon \text{ for each } f \in \C(L, [0, 1]) \text{ which satisfies } f(x)=1, \abs{Tf(y)}>\frac{1}{2}\Big\}.\]
%\[\Lambda_x=\Big\{y \in K \colon \text{ for each }f \in \C(L, [0, 1])\text{ if } f(x)=1\text{, then }\abs{Tf(y)}>{1}/{2}\Big\}.\]
\[\Lambda_x=\bigcap_{\substack{f \in \C(L, [0, 1])\\f(x)=1}}\big\{y \in K \colon \abs{T(f)(y)}>{1}/{2}\big\}.\]
By the proof of \cite[Theorem 1.7]{GalegoVillamizar}, for each ordinal $\alpha$, if $x \in L^{(\alpha)}$, then $\Lambda_x \cap K^{(\alpha)}\neq\emptyset$.%, so in particular $\Lambda_x\neq\emptyset$. 

Fix an ordinal $\alpha$. Let $\mathcal{U}$ be a cellular family in $L$ such that for every $U\in\mathcal{U}$ there is a point $x_U\in U\cap L^{(\alpha)}$. For each $U\in\mathcal{U}$ we find a function $f_U \in \C(L, [0, 1])$ such that $f_U(x_U)=1$ and $f_U$ vanishes outside of the set $U$, and put 
\[V_U=\big\{y \in K \colon \abs{T(f_U)(y)} > {1}/{2}\big\}.\]
Since $T$ is an isometric embedding, the sets $\{V_U\}_{U\in\mathcal{U}}$ are pairwise disjoint (cf. again the proof of \cite[Theorem 1.7]{GalegoVillamizar}). Also, for every $U\in\mathcal{U}$, we have $\Lambda_{x_U}\subseteq V_U$, so by the above argument it holds $V_U\cap K^{(\alpha)}\neq\emptyset$. This proves the required inequality.
\end{proof}

\subsection{Decomposition of continuous functions and one-point compactifications of disjoint unions}

%For a compact space $K$ and a function $f\in\C(K)$, we define the \emph{support} $\spt(f)$ of $f$ as usual: $\spt(f)=\overline{\{x\in K\colon f(x)\neq0\}}$. Also, for a subset $U\subseteq K$, we say that $f$ is \emph{supported} by $U$ if $\spt(f)\subseteq U$.

We will repeatedly use the following simple observation concerning families of functions. If $\{ U_{\gamma}\}_{\gamma \in \Gamma}$ is a family of pairwise disjoint nonempty subsets of a compact space $K$ and $\{g_{\gamma}\}_{\gamma \in \Gamma}$ are functions in $\C(K)$ such that, for each $\gamma \in \Gamma$, the function $g_{\gamma}$ is supported by the set $U_{\gamma}$, and we have $\{\norm{g_{\gamma}}\}_{\gamma \in \Gamma} \in c_0(\Gamma)$, then the sum $\sum_{\gamma \in \Gamma} g_{\gamma}$ is a well-defined element of $\C(K)$. Indeed, $\sum_{\gamma \in \Gamma} g_{\gamma}$ is a well-defined function on $K$, because at each point $x \in K$ there is at most one $\gamma \in \Gamma$ such that $g_{\gamma}(x) \neq 0$, and the function is continuous, since the condition $\{\norm{g_{\gamma}}\}_{\gamma \in \Gamma} \in c_0(\Gamma)$ implies that it belongs to the uniform closure of the family of continuous functions of the form $\sum_{\gamma \in F} g_{\gamma}$, where $F$ is a finite subset of $\Gamma$. 

%In the next lemma, we consider a function on a general compact space $K$, which is split in a certain way to a $c_0(\Gamma)$-sum of functions as in the previous paragraph, and summed up with one another function, which for us will play the role of a constant one, since, even if it need not be constant on entire $K$, it is constant on the supports of the functions appearing in the $c_0(\Gamma)$-sum. Later, 
%For a function $f \in \C(K)$ and a point $x \in K$, we standardly set $f^+(x)=\max \{f(x), 0\}$ and $f^-(x)=\max \{-f(x), 0\}$ (so $f^-(x)\ge0$). Of course, $f^+,f^-\in\C(K)$. Also, $\|f\|=\max\{\|f^+\|,\|f^-\|\}$, and we will frequently use the following explicit way of computing the norms of $f^+$ and $f^-$. 
We will frequently need the following lemma.

\begin{lemma}
\label{special form}
Let $K$ be a compact space, $a \in \er$, and $h \in \C(K, [0, 1])$ be a function of norm $1$. Further, let $\{ U_{\gamma} \}_{\gamma \in \Gamma}$ be an infinite family of pairwise disjoint open subsets of $K$ such that $U_{\gamma} \subseteq \{z \in K \colon h(z)=1\}$ for each $\gamma \in \Gamma$. Moreover, let $\{g_{\gamma}\}_{\gamma \in \Gamma}$ be a collection of functions in $\C(K)$ such that $\spt g_{\gamma} \subseteq U_{\gamma}$ for each $\gamma\in\Gamma$, and $\{ \norm{g_{\gamma}} \}_{\gamma \in \Gamma} \in c_0(\Gamma)$. Set
\[f=a h+\sum_{\gamma \in \Gamma} g_{\gamma}.\]
Then, $f\in\C(K)$, and we have
\[\norm{f^+}=\max \Big\{a+\sup_{\gamma \in \Gamma} \norm{(g_{\gamma})^+},\ 0\Big\} ,\]
\[\norm{f^-}=\max \Big\{-a+\sup_{\gamma \in \Gamma} \norm{(g_{\gamma})^-},\ 0\Big\},\]
and thus
\[\norm{f}=\max \Big\{a+\sup_{\gamma \in \Gamma} \norm{(g_{\gamma})^+},\ -a+\sup_{\gamma \in \Gamma} \norm{(g_{\gamma})^-} \Big\} .\]
\end{lemma}

\begin{proof}
By the discussion preceding the lemma, $f\in\C(K)$. We only prove the equality
\[\norm{f^+}=\max \Big\{a+\sup_{\gamma \in \Gamma} \norm{(g_{\gamma})^+},\ 0\Big\},\]
as the argument for the other one is analogous. 
Let $x \in K$. If $x \in U_{\gamma}$ for some $\gamma \in \Gamma$, then $h(x)=1$ and so
\[f^+(x)=\max \{f(x), 0\}=\max \big\{ah(x)+g_{\gamma}(x), 0\big\} \leq \max \Big\{a+\sup_{\gamma \in \Gamma} \norm{(g_{\gamma})^+},\ 0\Big\}.\]
If, on the other hand, $x \in K \setminus \bigcup_{\gamma \in \Gamma} U_{\gamma}$, then 
\[f^+(x)=\max \{ah(x), 0\}\leq \max \{a, 0\}\leq \max \Big\{a+\sup_{\gamma \in \Gamma} \norm{(g_{\gamma})^+},\ 0 \Big\}. \]
This proves the inequality $\leq$.

For the proof of the inequality $\geq$, let $C=\sup_{\gamma \in \Gamma} \norm{(g_{\gamma})^+}$.  We may assume that $a+C \geq 0$, otherwise the required inequality holds trivially. We distinguish two cases. 

If $C=0$, then $a \geq 0$. Further, for each $\ep>0$ there exists $\gamma \in \Gamma$ such that $\norm{g_{\gamma}}<\ep$. Thus, if we pick an arbitrary point $x \in U_{\gamma}$ (so $h(x)=1$), we get 
\[\norm{f^+}\ge f^+(x)=\max \{f(x), 0\} \geq \max \{ah(x)-\ep, 0\} \geq a+C-\ep,\]
which, since $\ep$ can be taken arbitrarily small, proves the required inequality. 

If $C>0$, then we fix $\gamma \in \Gamma$ and a point $x \in U_{\gamma}$ (so again $h(x)=1$) such that $C=g^+_{\gamma}(x)=g_{\gamma}(x)$. Then, 
%\[\norm{f^+}\ge f^+(x)=f(x)=ah(x)+g_{\gamma}(x)=a+C,\]
\[\norm{f^+}\ge f^+(x)=\max\{f(x),0\}=\max\{ah(x)+g_\gamma(x),0\}=\]
\[=\max\{a+C,0\}\ge a+C,\]
which finishes the proof of the inequality $\geq$.
\end{proof}

We now consider one-point compactifications of disjoint unions of families of compact spaces, for which we need to introduce a piece of notation. %We prove several general results about their spaces of continuous functions and about continuous surjections of such spaces, which we apply in the next subsection to the case of countable compact spaces., since countable compact spaces are naturally built up by the process of taking one-point compactifications of topological sums.
Let $\Gamma$ be a nonempty set and $\{ L_{\gamma}\}_{\gamma \in \Gamma}$ be a family of nonempty compact spaces. If $\Gamma$ is finite, then by $K_{\{L_{\gamma}\}_{\gamma \in \Gamma}}$ we denote the disjoint union of the spaces $L_{\gamma}$, and if $\Gamma$ is infinite, then $K_{\{L_{\gamma}\}_{\gamma \in \Gamma}}$ denotes the one-point compactification of the disjoint union of the spaces $L_{\gamma}$ with the new added point denoted simply by $\infty$. %, it should be always clear from the context to which space this point belongs. 
In the case when there is a compact space $L$ such that $L_{\gamma}=L$ for each $\gamma \in \Gamma$, we denote $K_{\{L_{\gamma}\}_{\gamma \in \Gamma}}$ simply by $K_{\Gamma, L}$. We identify the spaces $L_{\gamma}$ with subspaces of $K_{\{L_{\gamma}\}_{\gamma \in \Gamma}}$ in the natural way. Also, since the spaces $L_{\gamma}$ are clopen in the space $K_{\{L_{\gamma}\}_{\gamma \in \Gamma}}$, for each $\eta\in\Gamma$, we may naturally identify every function $f \in \C(L_{\eta})$ with the function from the space $\C(K_{\{L_{\gamma}\}_{\gamma \in \Gamma}})$ which coincides with $f$ on $L_{\eta}$ and is constantly $0$ on the rest of the space $K_{\{L_{\gamma}\}_{\gamma \in \Gamma}}$. This gives us the natural isometric embedding of $\C(L_\eta)$ into $\C(K_{\{L_{\gamma}\}_{\gamma \in \Gamma}})$. 

The following lemma describes a useful way how a continuous function on $K_{\{L_{\gamma}\}_{\gamma \in \Gamma}}$ can be naturally split in the spirit of Lemma \ref{special form} into a constant part and a certain $c_0(\Gamma)$-combination of functions defined on the spaces $L_{\gamma}$, $\gamma\in\Gamma$. For the case of ordinal intervals, this idea has already appeared in Rosenthal \cite[page 12]{RosenthalC(K)}.   

\begin{lemma}
\label{Form}
Let $\Gamma$ be an infinite set, $\{L_{\gamma}\}_{\gamma \in \Gamma}$ be a family of nonempty compact spaces, and $g \in \C(K_{\{L_{\gamma}\}_{\gamma \in \Gamma}})$. Then, $g$ may be uniquely written in the form 
    \[g=a_g\mathbf{1}+\sum_{\gamma \in \Gamma} g_{\gamma},\]
     where $a_g \in \er$, $\mathbf{1}$ is the constant-$1$ function on the space $K_{\{L_{\gamma}\}_{\gamma \in \Gamma}}$, for each $\gamma \in \Gamma$ we have $g_{\gamma}\in\C(L_{\gamma})$, and it holds $\{ \norm{g_{\gamma}} \}_{\gamma \in \Gamma} \in c_0(\Gamma)$.
\end{lemma}

\begin{proof}
We simply put $a_g=g(\infty)$ and, for each $\gamma \in \Gamma$, let  $g_{\gamma}=g|_{L_{\gamma}}-a_g \chi_{L_{\gamma}}$. From the continuity of $g$ at the point $\infty$ it follows that $\{\norm{g_{\gamma}}\}_{\gamma \in \Gamma} \in c_0(\Gamma)$. Moreover, the pairwise disjointness of the family $\{L_\gamma\}_{\gamma\in\Gamma}\cup\{\{\infty\}\}$ implies that this representation is unique.%it is easy to see that there is no other way to represent the function $g$ in this form, since $a$ \emph{must} be equal to $g(\infty)$. 
\end{proof}

To proceed further, we will need the following two auxiliary definitions.

\begin{definition}
For compact spaces $K$ and $L$, an open subset $U$ of $K$, and a linear embedding $T$ of $\C(L)$ into $\C(K)$, we say that $T$ is \emph{supported by} $U$ if $\spt(T(f)) \subseteq U$ for each $f \in \C(L)$.
\end{definition}

%Here, of course, for a compact space $K$ and a function $f \in \C(K)$, $\spt f$ stands for the support of $f$, that is, the closure of the set $\{x \in K\colon  f(x) \neq 0\}$.

\begin{definition}
If $K$, $L$ are compact spaces and $T \colon \C(L) \rightarrow \C(K)$ is an isometric embedding, we say that $T$ \emph{preserves norms of positive parts} (or, shortly, is \emph{PNPP}), if $\norm{f^+}=\norm{(T(f))^+}$ for each $f \in \C(K)$. %To shorten the notation, we simply say that the embedding is PNPP.
\end{definition}

Of course, if an isometric embedding $T\colon\C(L)\to\C(K)$ is PNPP, then it preserves also the norms of negative parts, that is, $\norm{f^-}=\norm{(T(f))^-}$ for each $f \in \C(L)$. Further, since a continuous function is positive if and only if the norm of its negative part is $0$, it follows that a PNPP embedding $T \colon \C(L) \rightarrow \C(K)$ is, in particular, a positive operator between the Banach lattices $\C(L)$ and $\C(K)$. On the other hand, each isometric embedding which is a Banach lattice isomorphism onto its image is PNPP. We will use these facts frequently.

The following lemma stands behind our motivation to introduce the above notion. It shows that PNPP isometric embeddings work nicely in the class of spaces $\C(K_{\{L_{\gamma}\}_{\gamma \in \Gamma}})$, which is not true for isometric embeddings which are not PNPP.

\begin{prop}
\label{copies of C(K_Gamma_L)}
Let $K$ be a compact space, $U$ be an open subset of $K$, and $h \in \C(K, [0, 1])$ be a function of norm $1$ supported by $U$. Further, for an infinite set $\Gamma$,  let $\{L_{\gamma}\}_{\gamma \in \Gamma}$ be a family of nonempty compact spaces, and $\{U_{\gamma}\}_{\gamma \in \Gamma}$ be a family of pairwise disjoint nonempty open subsets of $K$ such that $U_{\gamma} \subseteq \{x \in K \colon h(x)=1\}$ for each $\gamma \in \Gamma$. Assume also that for each $\gamma \in \Gamma$ there is a PNPP isometric embedding $T_\gamma\colon \C(L_{\gamma})\to\C(K)$ which is supported by $U_{\gamma}$.

Then, there is a PNPP isometric embedding $T\colon\C(K_{\{L_{\gamma}\}_{\gamma \in \Gamma}})\to\C(K)$ which is supported by $U$.
\end{prop}

\begin{proof}
We define the desired isometric embedding $T \colon \C(K_{\{L_{\gamma}\}_{\gamma \in \Gamma}}) \rightarrow \C(K)$ in the following way. Given $g \in \C(K_{\{L_{\gamma}\}_{\gamma \in \Gamma}})$, we write it in the form
\[g=a_g\mathbf{1}+\sum_{\gamma \in \Gamma} g_{\gamma}\]
as in Lemma \ref{Form}. %Thus, $a \in \er$, $f$ is the constant function $1$ on the space $K_{\{L_{\gamma}\}_{\gamma \in \Gamma}}$, for each $\gamma \in \Gamma$, $g_{\gamma}$ belongs to the space $\C(L_{\gamma})$, and it holds $\{ \norm{g_{\gamma}} \}_{\gamma \in \Gamma} \in c_0(\Gamma)$.
Then, we define 
\[T(g)= a_g h +\sum_{\gamma \in \Gamma} T_{\gamma}(g_{\gamma}).\]
First, we note that the mapping $T$ is well-defined, which follows from the fact that there is only one way to express each $g$ in the above form. Further, since 
$\{\norm{T_{\gamma}(g_\gamma)}\}_{\gamma \in \Gamma}=\{\norm{g_\gamma}\}_{\gamma \in \Gamma}\in c_0(\Gamma)$ and $\spt(T_\gamma(g_\gamma))\subseteq U_\gamma$ for each $\gamma\in\Gamma$, $T(g)$ belongs to the space $\C(K)$ by Lemma \ref{special form}. Also, it should be clear that $T$ is linear. 

It remains to check that $T$ is a PNPP isometric embedding. To this end, by using Lemma \ref{special form} twice, first in the space $\C(K_{\{L_{\gamma}\}_{\gamma \in \Gamma}})$ and then in the space $\C(K)$, for each $g\in\C(K_{\{L_{\gamma}\}_{\gamma \in \Gamma}})$ we have 
%\begin{equation}
%\nonumber
%\begin{aligned}
\[\norm{g}=\max \Big\{a_g+\sup_{\gamma \in \Gamma} \norm{(g_{\gamma})^+},\ -a_g+\sup_{\gamma \in \Gamma} \norm{(g_{\gamma})^-} \Big\}=\]  
%=\\&=
%\max \{a+\sup_{\gamma \in \Gamma} \norm{T_{\gamma}((g_{\gamma})^+)}, -a+\sup_{\gamma \in \Gamma} \norm{T_{\gamma}((g_{\gamma})^-)}\} \}  
%=\\&=
\[=\max \Big\{a_g+\sup_{\gamma \in \Gamma} \norm{(T_{\gamma}(g_{\gamma}))^+},\ -a_g+\sup_{\gamma \in \Gamma} \norm{(T_{\gamma}(g_{\gamma}))^-} \Big\}  
=
\norm{T(g)},\]
%\end{aligned}
%\end{equation}
and hence $T$ is an isometric embedding.

Further, $T$ is PNPP, since again by using Lemma \ref{special form} twice, for each $g\in\C(K_{\{L_{\gamma}\}_{\gamma \in \Gamma}})$ we get 
%\begin{equation}
%\nonumber
%\begin{aligned}
\[\norm{g^+}=\max \Big\{a_g+\sup_{\gamma \in \Gamma} \norm{(g_{\gamma})^+},\ 0\Big\}  =\]
%=\\&=
\[=\max \Big\{a_g+\sup_{\gamma \in \Gamma} \norm{(T_{\gamma}(g_{\gamma}))^+},\ 0\Big\}     =\norm{(Tg)^+}.\]
%\end{aligned}
%\end{equation}   
\end{proof}

Next, we show a way of constructing continuous surjection onto spaces $K_{\{L_{\gamma}\}_{\gamma \in \Gamma}}$.

\begin{lemma}
\label{Mapping onto K_Gamma_L}
Let $K$ be a compact space, $\Gamma$ be an infinite set, $\{L_{\gamma}\}_{\gamma \in \Gamma}$ be a family of nonempty compact spaces, and $\{U_{\gamma}\}_{\gamma \in \Gamma}$ be a family of pairwise disjoint clopen subsets of $K$  such that $U_{\gamma}$ maps continuously onto $L_{\gamma}$ for each $\gamma \in \Gamma$. Then, $K$ maps continuously onto $K_{\{L_{\gamma}\}_{\gamma \in \Gamma}}$. 
\end{lemma}

\begin{proof}
For each $\gamma \in \Gamma$ we fix a continuous surjection $\rho_{\gamma} \colon U_{\gamma} \rightarrow L_{\gamma}$. Let the mapping $\rho \colon K \rightarrow K_{\{L_{\gamma}\}_{\gamma \in \Gamma}}$ be defined for every $x \in K$ as
\[\rho(x)=\begin{cases} \rho_{\gamma}(x),& \text{if }x \in U_{\gamma}\text{ for some }\gamma\in\Gamma,\\
\infty,& \text{if }x \in K \setminus \bigcup_{\gamma \in \Gamma} U_{\gamma}.
\end{cases}\]
It is clear that $\rho$ is onto, thus we only need to check that it is continuous. To this end, we pick an arbitrary open subset $V$ of $K_{\{L_{\gamma}\}_{\gamma \in \Gamma}}$. Assuming that $\infty$ is not an element of $V$, we have
\[\rho^{-1}[V]=\rho^{-1}\Big[\bigcup_{\gamma \in \Gamma} (V \cap L_{\gamma})\Big]= \bigcup_{\gamma \in \Gamma} \rho_{\gamma}^{-1}[V \cap L_{\gamma}],\]
so $\rho^{-1}[V]$ is an open subset of $K$. 
If, on the other hand, $\infty$ belongs to $V$, we find a finite set $F \subseteq \Gamma$ such that $L_\gamma \subseteq V$ for each $\gamma \in \Gamma \setminus F$. Then, we have 
%\begin{equation}
%\nonumber
%\begin{aligned}
\[\rho^{-1}[V]=\rho^{-1}\Big[\{\infty\} \cup \bigcup_{\gamma \in \Gamma} (V \cap L_{\gamma})\Big]=\]
%\[=\rho^{-1}(\{\infty\} \cup \bigcup_{\gamma \in F} (V \cap L_{\gamma}) \cup \bigcup_{\gamma \in \Gamma \setminus F} (V \cap L_{\gamma}))=\]
\[=\rho^{-1}(\infty) \cup \bigcup_{\gamma \in F} \rho^{-1}[V \cap L_{\gamma}] \cup \bigcup_{\gamma \in \Gamma \setminus F} \rho^{-1}[L_{\gamma}]=\]
\[=\big(K \setminus \bigcup_{\gamma \in \Gamma} U_{\gamma}\big) \cup \bigcup_{\gamma \in F} \rho_{\gamma}^{-1}[V \cap L_{\gamma}] \cup \bigcup_{\gamma \in \Gamma \setminus F} U_{\gamma}=\]
\[=\big(K \setminus \bigcup_{\gamma \in F} U_{\gamma}\big) \cup \bigcup_{\gamma \in F} \rho_{\gamma}^{-1}[V \cap L_{\gamma}],\]
%\end{aligned}
%\end{equation}
so $\rho^{-1}[V]$ is again an open subset of $K$ (recall that each $U_\gamma$ is clopen). The proof is thus finished. 
\end{proof}

\subsection{Countable compact spaces}

In this subsection we first show how countable compact spaces are naturally inductively built by  taking one-point compactifications of disjoint unions of ordinal intervals, which will later allow us to apply the results from the previous subsection. 

We start with the following lemma.
%Before proceeding with the first lemma, let us recall that, by the classical Mazurkiewicz--Sierpi\'{n}ski classification \cite[Proposition 8.6.5]{semadeni}, each countable compact space is homeomorphic to the interval $[1, \omega^{\alpha}m]$ for some countable ordinal $\alpha$ and $m \in \en$.

\begin{lemma}
\label{description of omega^alpha}
Let $\alpha$ be an ordinal.
\begin{enumerate}[(i)]
    \item If $\alpha=\beta+1$ is a successor ordinal, then the space $[1, \omega^{\alpha}]$ is homeomorphic to $K_{\en, [1, \omega^{\beta}]}$.
    \item If $\alpha$ is a limit ordinal and $(\beta_n)_{n \in \en}$ is an increasing sequence of ordinals converging to $\alpha$, then the space $[1, \omega^{\alpha}]$ is homeomorphic to $K_{\{[1, \omega^{\beta_n}]\}_{n \in \en}}$.
\end{enumerate}
\end{lemma}

\begin{proof}
(i) If $\alpha=\beta+1$, we fix for each $n \in \en$ a (surjective) homeomorphism
\[\rho_n \colon\big[1, \omega^{\beta}\big] \rightarrow \big[\omega^{\beta}(n-1)+1, \omega^{\beta}n\big].\]
Then, we define the mapping $\rho \colon K_{\en, [1, \omega^{\beta}]} \rightarrow [1, \omega^{\alpha}]$ for every $x\in K_{\en, [1, \omega^{\beta}]}$ by
\[\rho(x)=\begin{cases} \rho_n(x),&\text{ if } x \text{ lies in the } n\text{-th copy of } \big[1, \omega^{\beta}\big] \text{ in the space } K_{\en, [1, \omega^{\beta}]},\\
\omega^{\alpha},&\text{ if }x=\infty.
\end{cases}\]
It is clear that $\rho$ is surjective and it is elementary to check that it is a homeomorphism (cf. the proof of Lemma \ref{Mapping onto K_Gamma_L}).

(ii) If $\alpha$ is a limit ordinal and $(\beta_n)_{n \in \en}$ is an increasing sequence of ordinals converging to $\alpha$, we additionally define $\beta_0=0$ and for each $n \in \en$ we fix a (surjective) homeomorphism
\[\rho_n \colon \big[1, \omega^{\beta_n}\big] \rightarrow \big[\omega^{\beta_{n-1}}+1, \omega^{\beta_n}\big]\]
(note that all the above spaces are homeomorphic by the Mazurkiewicz--Sierpi\'{n}ski theorem). Analogously as above, we define
\[\rho \colon K_{\{[1, \omega^{\beta_n}]\}_{n \in \en}} \rightarrow \big[1, \omega^{\alpha}\big]\]
for every $x\in K_{\{[1, \omega^{\beta_n}]\}_{n \in \en}}$ by
\[\rho(x)=\begin{cases} \rho_n(x),&\text{ if } x \in \big[1, \omega^{\beta_n}\big]\text{ for some }n\in\en,\\
\omega^{\alpha},&\text{ if }x=\infty,
\end{cases}\]
and it is again simple to check that $\rho$ is a surjective homeomorphism.
\end{proof}

We are ready to prove the key ingredient of the proof of our main result about isometric embeddings, Theorem \ref{main-isometric}.

\begin{prop}
\label{Copies}
Let $K$ be a compact space, $U$ be an open subset of $K$, $m \in \en$, and $\alpha$ be a countable ordinal such that $\abs{U^{(\alpha)}} \geq m$. Then, there exists a PNPP isometric embedding $T\colon\C([1, \omega^{\alpha}m])\to\C(K)$ which is supported by $U$.  
\end{prop}

\begin{proof}
We first assume that $m=1$ and prove this case by transfinite induction. The case of $\alpha=0$ is simple: it is enough to find a function $h \in \C(K, [0, 1])$, of norm $1$ and supported by $U$, and then  to consider the mapping $T \colon \C(\{1\}) \cong \er \rightarrow \C(K)$ defined for every $a\in\er$ by $T(a)=a h$. 

So, let us assume that $\alpha$ is a nonzero ordinal and that the statement holds for each $\beta < \alpha$. We fix a point $y \in U^{(\alpha)}$ and let $V$ be an open neighborhood of $y$ such that $\overline{V}\subseteq U$. Let $h \in \C(K, [0, 1])$ be a function such that
\[\overline{V}\subseteq\{z \in K \colon h(z)=1\}\subseteq\spt h \subseteq U.\]
Note that $y\in V\cap U^{(\alpha)}\subseteq V^{(\alpha)}$, so $ht(V)>\alpha$.

If $\alpha=\beta+1$ is a successor ordinal, we use Lemma \ref{Split} to find pairwise disjoint nonempty open sets $(U_n)_{n \in \en}$ such that $ht(U_n)\ge\alpha$ and $\overline{U_n}\subseteq V$ for each $n\in\en$. Next, by the inductive assumption, for each $n\in\en$, we find a PNPP isometric embedding $T_n\colon\C([1, \omega^{\beta}])\to\C(K)$ which is supported by $U_n$. Then, an application of Proposition \ref{copies of C(K_Gamma_L)} yields a PNPP isometric embedding $S\colon\C(K_{\en, [1, \omega^{\beta}]})\to\C(K)$ supported by $V$. Since $\C(K_{\en, [1, \omega^{\beta}]})$ is lattice isometric to $\C([1, \omega^{\alpha}])$ by Lemma \ref{description of omega^alpha}, composing the isometries yields a PNPP isometric embedding $T\colon\C([1, \omega^{\alpha}])\to\C(K)$ supported by $V$ and so by $U$, and this finishes the successor case.

If $\alpha$ is a limit ordinal, we proceed similarly as above. First, we fix an increasing sequence of ordinals $(\beta_n)_{n \in \omega}$ converging to $\alpha$ (which exists as $\alpha$ is countable), with $\beta_0=0$. By Lemma \ref{Split} we can find pairwise disjoint nonempty open sets $(U_n)_{n \in \en}$ such that $ht(U_n)\ge\beta_n+1$ and $\overline{U_n}\subseteq V$ for each $n\in\en$. By the inductive assumption, for each $n \in \en$, we find a PNPP isometric embedding $T_n\colon\C([1, \omega^{\beta_n}])\to\C(K)$ which is supported by $U_n$. Again by Proposition \ref{copies of C(K_Gamma_L)}, we obtain a PNPP isometric embedding $S\colon\C(K_{\{[1, \omega^{\beta_n}]\}_{n \in \en}})\to\C(K)$ supported by $V$. Since $\C(K_{\{[1, \omega^{\beta_n}]\}_{n \in \en}})$ is lattice isometric to $\C([1, \omega^{\alpha}])$ again by Lemma \ref{description of omega^alpha}, composing the isometries we get a PNPP isometric embedding $T\colon\C([1, \omega^{\alpha}])\to\C(K)$ which is supported by $V$ and so by $U$, and thus the limit case is also done.

Finally, if $m>1$, we fix distinct points $y_1, \ldots, y_m  \in U^{(\alpha)}$ and we further find pairwise disjoint open sets $(U_i)_{i=1}^m$ such that $y_i\in U_i\subseteq U$ for each $i=1,\ldots,m$. Then, by the above argument, for each $i=1,\ldots,m$, there exists a PNPP isometric embedding 
\[T_i\colon\C([1, \omega^{\alpha}]) \cong \C([\omega^{\alpha}(i-1)+1, \omega^{\alpha}i])\to\C(K)\]
which is supported by $U_i$. It is then clear that there exists a PNPP isometric embedding $T\colon\C([1, \omega^{\alpha}m])\to\C(K)$ which is supported by $U$.
\end{proof}

\begin{remark}
In the process of proving Proposition \ref{Copies} we made use of the fact that, for each countable ordinal $\alpha$, the space $\C([1, \omega^{\alpha}])$ can be naturally viewed as a $c_0$-combination of spaces of continuous functions on smaller intervals, plus the constant part (cf. Lemmas \ref{Form} and \ref{description of omega^alpha}). Another possibile argument for the proof would be to use the tree description of the space $\C([1, \omega^{\alpha}])$, which has been considered, for example, by Bourgain \cite[Section 2]{Bourgain1979} (see also \cite[page 35]{RosenthalC(K)}), in which case we would show that the closed linear span of the functions which we inductively constructed in Proposition \ref{Copies} is isometric to a suitable tree. That approach seems to require a proof of similar difficulty as the one given above, but it would have a disadvantage of being a little bit less transparent, since, e.g., it would rely on \cite[Lemma 4]{Bourgain1979}, which is stated without a proof. 
\end{remark}

The key ingredient of the proof of Theorem \ref{main-isometric-zero} reads as follows. 

\begin{prop}
\label{ContinuousMaps}
Let $K$ be a zero-dimensional compact space, $U$ be a clopen subset of $K$, $m \in \en$, and $\alpha$ be a countable ordinal such that $\abs{U^{(\alpha)}} \geq m$. Then, $U$ maps continuously onto $[1, \omega^{\alpha}m]$.
\end{prop}

\begin{proof}
We first assume that $m=1$ and we proceed by transfinite induction. Trivially, if $U$ is nonempty, then it maps continuously onto $\{1\}$, so we may assume that $\alpha>0$ and that the statement holds for every $\beta<\alpha$.

We first assume that $\alpha=\beta+1$ is a successor ordinal. Since $U^{(\alpha)}\neq\emptyset$, by Lemma \ref{Split}.(i) we may find pairwise disjoint clopen subsets $(U_n)_{n \in \en}$ of $U$ such that $U_n^{(\beta)} \neq \emptyset$ for each $n \in \en$. By the inductive assumption, for each $n \in \en$,  $U_n$ maps continuously onto $[1, \omega^{\beta}]$. Since $K$ is compact,  $K\neq\bigcup_{n\in\en}U_n$. Thus, $U$ maps continuously onto $K_{\en, [1, \omega^{\beta}]}$ by Lemma \ref{Mapping onto K_Gamma_L} and so onto $[1, \omega^{\alpha}]$ by  Lemma \ref{description of omega^alpha}.(i).

Likely, if $\alpha$ is a limit ordinal, we fix an increasing sequence of ordinals $(\beta_n)_{n \in \en}$ converging to $\alpha$ (which exists as $\alpha$ is countable), and we use Lemma \ref{Split}.(ii) to find pairwise disjoint nonempty clopen subets $(U_n)_{n \in \en}$ of $U$ such that $ht(U_n) \geq \beta_n+1$ for each $n \in \en$. By the assumption, for each $n \in \en$, $U_n$ maps continuously onto $[1, \omega^{\beta_{n}}]$. Also, $K\neq\bigcup_{n\in\en}U_n$. Then, as above, by Lemmas \ref{Mapping onto K_Gamma_L} and \ref{description of omega^alpha}, $U$ maps continuously onto $[1, \omega^{\alpha}]$.

Finally, if $m>1$, we fix distinct points $y_1, \ldots, y_m  \in U^{(\alpha)}$ and we further find for each $i=1, \ldots m$ a clopen subset $U_i$ of $U$ containing $y_i$ (so $U_i^{(\alpha)}\neq\emptyset$) and such that the sets $(U_i)_{i=1}^m$ are pairwise disjoint and $\bigcup_{i=1}^{m} U_i=U$. By the above argument, for each $i=1, \ldots, m$, we find a continuous surjection $\rho_i\colon U_i\to[\omega^{\alpha}(i-1)+1, \omega^{\alpha}i]$. %Now, it is enough to fix an arbitrary point $z \in [1, \omega^{\alpha}m]$ and to define the continuous surjection $\rho \colon U \rightarrow [1, \omega^{\alpha}m]$ for every $x\in U$ by 
%\[\rho(x)=\begin{cases} \rho_i(x),&\text{ if }x \in U_i\text{ for some }i=1,\ldots,m,\\
%z,&\text{ if }x \in U \setminus \bigcup_{i=1}^{m} U_i.
%\end{cases}\]
%The proof is finished.
Now, we define the continuous surjection $\rho \colon U \rightarrow [1, \omega^{\alpha}m]$ for every $x\in U$ by setting $\rho(x)=\rho_i(x)$ if $x \in U_i$ for some $i=1,\ldots,m$. The proof is finished.
\end{proof}

\section{The main results\label{sec:main}}

We are in the position to prove the main results of this paper. We start with the characterizations of the presence of an isometric copy of $\C(L)$ in $\C(K)$, when $L$ is metrizable---Theorems \ref{main-isometric} and \ref{main-isometric-zero}. The main items in these characterizations are (ii), (iii), (iv), and (ix), since they 
reduce the question of isometric embeddings to simple topological relations between the considered compact spaces. However, the other items in the characterizations are also of some independent interest on their own.

\begin{thm}
   \label{main-isometric}
For a compact metric space $L$ and a compact space $K$, the following assertions are equivalent:
\begin{enumerate}[(i)]
    \item $\C(K)$ contains an isometric copy of $\C(L)$,
    \item there exists a closed subset of $K$ which maps continuously onto $L$,
    \item $\abs{L^{(\alpha)}} \leq \abs{K^{(\alpha)}}$ for each ordinal $\alpha$,
    %\item $ht(L) \leq ht(K)$, and if $L$ is scattered, then $\abs{L^{(ht(L)-1)}} \leq \abs{K^{(ht(L)-1)}}$,
    \item if $L$ is nonscattered, then $K$ is nonscattered, and if $L$ is scattered, then $\abs{L^{(ht(L)-1)}} \leq \abs{K^{(ht(L)-1)}}$,
    \item $\C(K)$ contains a positive isometric copy of $\C(L)$,
    \item there is an isomorphic embedding $T \colon \C(L) \rightarrow \C(K)$ with $\norm{T} \norm{T^{-1}} <3$,
    \item for each Banach space $E$, $\C(K, E)$ contains an isometric copy of $\C(L)$,
    \item there exists a strictly convex Banach space $E$ such that $\C(K, E)$ contains an isometric copy of $\C(L)$.
\end{enumerate}
\end{thm}

\begin{figure}\[\xymatrixcolsep{4pc}\xymatrix{
\\
\mathrm{\,i\,}\ar@{:>}[r]\ar@{:>}@<-2pt>[dr]&    \mathrm{\,vii\,}\ar@{:>}[r]&\mathrm{\,viii\,}\ar@{=>}[r]^{\text{\ Cambern \cite{Cambern-vector_Holsztynski}}}&   \mathrm{ii}\ar@{=>}[r]^{\text{Corollary \ref{celularity_cor}\ \ \vphantom{[}}}&     \mathrm{iv}\ar@{=>}[r]\ar@{=>}[d]&     \mathrm{v}\ar@{:>}@/_3pc/@<-2pt>[lllll]\\
&     \mathrm{vi}\ar@{=>}[r]^{\text{\,Gordon \cite{Gordon3} }}&    \mathrm{iii}\ar@{:>}@<-2pt>[urr]&     &   \mathrm{ix}\ar@{:>}[r]&     \mathrm{x}\ar@{:>}[u]\\%\ar@{:>}@/_3pc/[llll]\\
}\]
\vspace{0mm}
\begin{center}\captionof{figure}[aaaaaa]{The scheme of the proofs of Theorems \ref{main-isometric} and \ref{main-isometric-zero}. Dashed arrows indicate trivial implications. \label{diagram}}
\end{center}\end{figure}
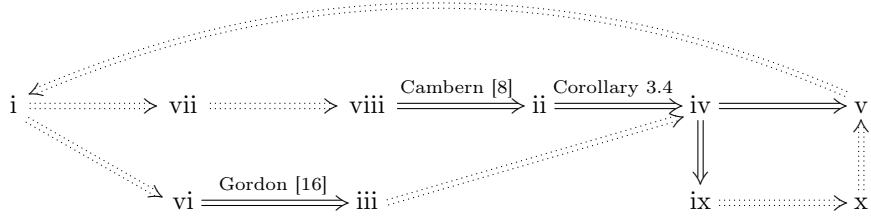

\begin{proof}
For the scheme of the proof, see Figure \ref{diagram}. %We assume that $K$ is infinite, otherwise the theorem is obvious.

Implications (v)$\Rightarrow$(i), (i)$\Rightarrow$(vi), (i)$\Rightarrow$(vii), (vii)$\Rightarrow$(viii), and (iii)$\Rightarrow$(iv) are trivial. 

As mentioned in the introduction, implication (vi)$\Rightarrow$(iii) follows from the theorem of Gordon \cite{Gordon3}. Also, (viii)$\Rightarrow$(ii) follows from Cambern's generalization \cite{Cambern-vector_Holsztynski} of the Holszty\'{n}ski theorem, asserting that if $\C(L)$ is isometrically embedded into the space $\C(K, F)$, where $F$ is a strictly convex Banach space, then $L$ is a continuous image of a closed subset of $K$. 

(ii)$\Rightarrow$(iv) follows from Corollary \ref{celularity_cor}.%Proposition \ref{celularity}.

(iv)$\Rightarrow$(v) First, we assume that $L$ is not scattered, so $K$ is also not scattered. Thus, $K$ maps continuously onto $[0, 1]$, and hence $\C(K)$ contains a positive isometric copy of $\C([0, 1])$. Further, by Rosenthal \cite[Theorem 2.8]{RosenthalC(K)} (a result based on the proof of the classical Miljutin theorem),  $\C([0, 1])$ contains a positive isometric copy of $\C(L)$. Hence, $\C(K)$ contains a positive isometric copy of $\C(L)$. 

If, on the other hand, $L$ is scattered, then it is countable, so, by the  Mazurkiewicz--Sierpi\'{n}ski classification, $L$ is homeomorphic to $[1, \omega^{\alpha}m]$, where $\alpha=ht(L)-1$ is countable and $m=\abs{L^{(\alpha)}}$. Thus, $\C(L)$ is lattice isometric to $\C([1, \omega^{\alpha}m])$. Hence, if $\abs{K^{(\alpha)}} \geq \abs{L^{(\alpha)}}=m$, then, by Proposition \ref{Copies} (for $U=K$), $\C(K)$ contains a positive isometric copy of $\C([1, \omega^{\alpha}m])$ and thus a positive isometric copy of $\C(L)$. 
\end{proof}

The following theorem provides additional equivalent conditions in the case when $K$ is zero-dimensional.

\begin{thm}\label{main-isometric-zero}
For a compact metric space $L$ and a compact zero-dimensional space $K$, assertions (i)--(viii) of Theorem \ref{main-isometric} are additionally equivalent to the following ones:
\begin{enumerate}[(i)]\setcounter{enumi}{8}
    \item $K$ maps continuously onto $L$,
    \item there exists an isometric embedding of $\C(L)$ into $\C(K)$ which is also an embedding in the sense of Banach lattices and Banach algebras.
\end{enumerate}
\end{thm}
\begin{proof}
Implications (ix)$\Rightarrow$(x) and (x)$\Rightarrow$(Theorem \ref{main-isometric}.(v)) are trivial. We show implication (Theorem \ref{main-isometric}.(iv))$\Rightarrow$(ix). To this end, if $L$ is not scattered, then $K$ is also not scattered, so it follows that $K$ maps continuously onto the Cantor space $\Delta$. And it is well-known that $\Delta$ maps continuously onto any  metrizable compact space, in particular, onto $L$.  %so, composing the isometries, we consequently get an isometric embedding $T\colon\C(L)\to\C(K)$.

If $L$ is scattered, then, by the Mazurkiewicz--Sierpi\'{n}ski theorem, $L$ is homeomorphic to $[1, \omega^{\alpha}m]$, where $\alpha=ht(L)-1$ is countable and $m=\abs{L^{(\alpha)}}$, and so the assertion follows by Proposition \ref{ContinuousMaps} (for $U=K$), as $\abs{K^{(\alpha)}} \geq \abs{L^{(\alpha)}}=m$.
\end{proof}

We continue with several remarks concerning Theorems \ref{main-isometric} and \ref{main-isometric-zero}. The first one pertains to the proof of Theorem \ref{main-isometric}.

\begin{remark}\label{remark:first}
    (1) If the space $K$ is nonscattered, then item (v) always holds for any $L$, regardless of whether $L$ is scattered or not. Indeed, if $K$ is nonscattered, then, by Pe\l czy\'{n}ski--Semadeni theorem, $\C(K)$ contains a positive isometric copy of $\C([0,1])$ and therefore, by Rosenthal \cite[Theorem 2.8]{RosenthalC(K)} (a result based on the proof of the Miljutin theorem), a positive isometric copy of $\C(L)$ for any metrizable compact space $L$. Consequently, if $K$ is nonscattered, then all items of Theorem \ref{main-isometric} hold for any compact metric space $L$. Thus, as we mentioned in the introduction, the essential weight of the proof and the novelty of the theorem lie in the case when $K$ is scattered, note however that our methods are general and do not put any assumptions on $K$.

    (2) Note that neither Proposition \ref{celularity} nor Corollary \ref{celularity_cor} put any assumptions on the space $L$, in particular, $L$ may be nonmetrizable. Implication (ii)$\Rightarrow$(iv) holds therefore for any compact space $L$. In fact, the only place in the proof of Theorem \ref{main-isometric} where the assumption of metrizability of $L$ is really used is implication (iv)$\Rightarrow$(v), as this implication may not hold if $L$ is nonmetrizable. Indeed, consider $L=[1,\omega_1]$ and $K=[0,1]$ and note that $L$ has uncountable cellularity whereas $K$ satisfies the countable chain condition, which implies that $\C(L)$ cannot be even isomorphically embedded into $\C(K)$, by \cite[Theorem 1.3]{rondos_cardinal_invariants}.

    (3) Due to Proposition \ref{celularity} and Corollary \ref{celularity_cor}, implication (ii)$\Rightarrow$(iv) can be divided into two parts (ii)$\Rightarrow$(ii') and (ii')$\Rightarrow$(iv), where item (ii') reads as follows:
    
    \medskip
    \textit{(ii') there exists a closed subset $F$ of $K$ such that $c(L^{(\alpha)},L)\le c(F^{(\alpha)},F)$ for each ordinal $\alpha$.}
\end{remark}

The second remark discusses possible extensions of Theorems \ref{main-isometric} and \ref{main-isometric-zero}.

\begin{remark}
(1) It is possible to get results analogous to Theorem \ref{main-isometric} for complex Banach spaces of complex-valued functions on zero-dimensional compact spaces. For example, one can use \cite[Theorem 1.1.(c)]{rondos-scattered-subspaces} to obtain a characterization of the presence of an embedding of $\C(L,\mathbb{C})$ into $\C(K,\mathbb{C})$ in the sense of $C^*$-algebras.

(2) Let us note that there are some other assertions which could be added to Theorem \ref{main-isometric} and which would be analogous to item (vi). For example, for certain Banach spaces $E$, there exists $\lambda>1$ such for each compact spaces $L$ and $K$ the existence of an isomorphic embedding of $\C(L)$ into $\C(K,E)$ with distortion $<\lambda$ implies that $L$ is a continuous image of a closed subset of $K$, see e.g. \cite[Theorem 3.1]{Galego-Villamizar_continuous_maps}. An analogous result for positive embeddings and certain Banach lattices is contained in \cite[Theorem 1.1]{Galego_Villamizar_positive_embeddings}. Also, it would be possible to add conditions concerning certain nonlinear embeddings, see e.g. \cite{GalegoSilvaNonlinear}. However, it seems that trying to add all those statements into Theorem \ref{main-isometric} would make the list a bit too long and therefore illegible.

(3) Let us also mention that \cite[Theorem 2]{Benyamini_IntoIsomorphisms} shows, among other things, that if $L$ is a compact metric space and $K$ is a compact space such that there exists an isomorphic embedding $T \colon \C(L) \rightarrow \C(K)$ with $\norm{T} \norm{T^{-1}} <2$, then $\C(L)$ is already isometrically embedded into $\C(K)$. Our Theorem \ref{main-isometric} shows that the value $2$ can actually be replaced by $3$. %In the same paper \cite{Benyamini_IntoIsomorphisms}, it is also shown that the result is no longer true when $L$ is not metrizable, which also proves that items (i) and (vi) of Theorem \ref{main-isometric} are in general not equivalent without the assumption of metrizability of $L$.
\end{remark}

Now, let us comment on the possibility of reversing implications considered in the proof of Theorem \ref{main-isometric} as presented in Figure \ref{diagram}, when the assumption of metrizability of $L$ is removed. We focus only on implications between those items which also occur in Theorem \ref{thm:mainA} in Introduction, that is, items (i)--(vi).

\begin{remark}\label{remark_implications}
    For implications other than (i)$\Rightarrow$(v) we have:

    \medskip
    
    %(vii)$\Rightarrow$(i) The implication simply holds by considering $E=\er$.\\[2pt]
    %\indent (viii)$\not\Rightarrow$(vii) {\color{red}???????????????????????}\\[2pt]
    \indent (ii)$\not\Rightarrow$(i) \& (iii)$\not\Rightarrow$(vi) The counterexamples were discussed in Introduction.\\[2pt]
    %%%INTRO%%%\indent (ii)$\not\Rightarrow$(i) This was mentioned in Introduction. Put $L=[1,\omega_1]$ and $K=[0,1]^{\omega_1}$, so $L$ embeds homeomorphically into $K$. Consequently, (ii) holds. On the other hand, $K$ satisfies the countable chain condition whereas $L$ has uncountable cellularity, so again there is no isomorphic embedding of $\C(L)$ into any $\C(K,E)$-space, by \cite[Theorem 1.3]{rondos_cardinal_invariants}.\\[2pt]
    \indent (iv)$\not\Rightarrow$(ii) Set $K=[0,1]$ and $L=[1,\omega_1]$. Then, $L$ is scattered, $K^{(\infty)}=K$, and $\abs{L^{(ht(L)-1)}}=\abs{\{\omega_1\}}=1<\mathfrak{c}=\abs{K^{(\infty)}}=\abs{K^{(ht(L)-1)}}$, hence (iv) holds. However, no closed subset of $K$ maps onto $L$, as $L$ has uncountable cellularity, whereas each subset of $K$ satisfies the countable chain condition. Consequently, (ii) does not hold.\\[2pt]
    \indent (v)$\Rightarrow$(iv) Analyse Figure \ref{diagram}, having in mind that Corollary \ref{celularity_cor} does not require $L$ to be metrizable.\\[2pt]%Assume that, for two compact spaces $L$ and $K$, $\C(L)$ embeds positively isometrically into $\C(K)$. If $L$ is nonscattered, then $\C([0,1])$ embeds isometrically into $\C(L)$, hence into $\C(K)$, and so $K$ is also nonscattered. If, on the other hand, $L$ is scattered, then, by the aforementioned theorem of Gordon \cite{Gordon3} for $\alpha=ht(L)-1$, we have $\abs{L^{(ht(L)-1)}} \leq \abs{K^{(ht(L)-1)}}$.\\[2pt]
    %\indent (i)$\not\Rightarrow$(v) {\color{red}???????????????????????}\\[2pt]
    \indent (vi)$\not\Rightarrow$(i) For the counterexample, see \cite{Benyamini_IntoIsomorphisms}.\\[2pt]
    %%%INTRO%%%\indent (iii)$\not\Rightarrow$(vi) This was mentioned in Introduction. Let $L=[1,\omega_1]$ and $K=[0,1]$. Then, $K=K^{(\infty)}$, and so, for any ordinal $\alpha$, we have $\abs{L^{(\alpha)}}\le\mathfrak{c}=\abs{K}=\abs{K^{(\alpha)}}$, so (iii) is satisfied, but $\C(L)$ is not separable, hence cannot be isomorphically embeded into $\C([0,1])$, so (vi) does not hold.\\[2pt]
    \indent (iv)$\not\Rightarrow$(iii) For the nonscattered case, consider $L=(\beta\en)\setminus\en$, i.e. the remainder of the \v{C}ech--Stone compactification of $\en$, and $K=[0,1]$. Then, $L^{(\infty)}=L$ and $K^{(\infty)}=K$, so (iv) holds, but, for any ordinal $\alpha$, we have $\abs{L^{(\alpha)}}=\abs{L}=2^\mathfrak{c}>\mathfrak{c}=\abs{K}=\abs{K^{(\alpha)}}$, so (ii) does not hold.

    For the scattered case, let $L=[1,2^{\mathfrak{c}}]$, where $2^{\mathfrak{c}}$ denotes the cardinal exponentiation, and $K=[0,1]$. Then, $K=K^{(\infty)}=K^{(ht(L)-1)}$ and $L^{(ht(L)-1)}=\{2^{\mathfrak{c}}\}$, so $\abs{L^{(ht(L)-1)}}=1<\mathfrak{c}=\abs{K}=\abs{K^{(ht(L)-1)}}$, hence (iv) holds. On the other hand, for $\alpha=0$ we have $\abs{L^{(\alpha)}}=\abs{L}=2^{\mathfrak{c}}>\mathfrak{c}=\abs{K}=\abs{K^{(\alpha)}}$, so (ii) is not satisfied.
\end{remark}

%\begin{remark}\label{rem:impl_i_v}
%    As mentioned in Introduction, we do not know whether implication (i)$\Rightarrow$(v) holds for all compact spaces $L$ and $K$.
%\end{remark}

Concerning the validity of implication (i)$\Rightarrow$(v) for arbitrary spaces $L$, we pose the following question.

\begin{question}\label{question}
    Let $K$ and $L$ be compact spaces such that $\C(L)$ isometrically embeds into $\C(K)$. Suppose that $L$ is nonmetrizable. Does there exist a positive isometric embedding of $\C(L)$ into $\C(K)$?
\end{question}

The last remark concerns the necessity of the assumption of zero-dimensionality for the space $K$ in Theorem \ref{main-isometric-zero}.

\begin{remark}
Neither of items (ix) and (x) is equivalent with the assertions of Theorem \ref{main-isometric}, if the assumption of zero-dimensionality of $K$ is dropped. This is easy to see for the case of (ix)---consider simply the spaces $K=[0,1]$ and $L=\Delta$, and note that $\C(L)$ embeds positively isometrically into $\C(K)$ by Rosenthal \cite[Theorem 2.8]{RosenthalC(K)}, but $K$ cannot be continuously mapped onto $L$, since $L$ is disconnected. The case of (x) is more difficult, but \cite[Theorem C]{Aviles-Cervantes-RuedaZoca-Tradacete_linearvslatticeembeddings} shows that (i) and (x) are also not equivalent if $K$ is not zero-dimensional.
\end{remark}

%\medskip

We now move to our next main result, Theorem \ref{main-isomorphic}, which characterizes the presence of isomorphic copies of $\C(L)$ in $\C(K)$, where $L$ is metrizable. As mentioned in the introduction, here the main equivalent condition is the inequality of the Szlenk indices of $\C(L)$ and $\C(K)$, $Sz(\C(L)) \leq Sz(\C(K))$. While this inequality is, strictly speaking, a relation between the Banach spaces instead of the underlying compact spaces, it might be also understood as a topological relation between the compact spaces, having in mind that the Szlenk index of $\C(K)$ can be computed directly from the height of $K$. To explain how the latter is done, we need first to recall that a \emph{gamma number} is either an ordinal which is either $0$ or of the form $\omega^{\beta}$ for some ordinal $\beta$, or $\infty$. For an ordinal $\alpha$, let $\Gamma(\alpha)$ denote the minimal gamma number which is not less than $\alpha$. Set also $\Gamma(\infty)=\infty$. Then, the Szlenk index of $\C(K)$ is simply $\Gamma(ht(K))$, by Causey \cite{CAUSEY_C(K)_index}.

\begin{thm}
\label{main-isomorphic}
For a compact metric space $L$ and a compact space $K$, the following assertions are equivalent:
\begin{enumerate}[(i)]
    \item $Sz(\C(L)) \leq Sz(\C(K))$,
    \item $\C(K)$ contains an isomorphic copy of $\C(L)$,
    \item for each Banach space $E$, $\C(K, E)$ contains an isomorphic copy of $\C(L)$,
    \item there exists a Banach space $E$ not containing an isomorphic copy of $c_0$ and such that $\C(K, E)$ contains an isomorphic copy of $\C(L)$,
    \item there exists a Banach space $E$ with $Sz(E) \leq Sz(\C(K))$ and such that $\C(K, E)$ contains an isomorphic copy of $\C(L)$.
\end{enumerate}
\end{thm}

\begin{proof}
Implications (ii)$\Rightarrow$(iii)$\Rightarrow$(iv) and (iii)$\Rightarrow$(v) are trivial.

(i)$\Rightarrow$(ii) We first suppose that $L$ is countable. Then, by the classical Bessaga--Pe\l czy\'{n}ski classification \cite{BessagaPelcynski_classification}, there exists a countable ordinal $\alpha$ such that $\C(L)$ is isomorphic to $\C\big(\big[1, \omega^{\omega^{\alpha}}\big]\big)$. Thus, \[Sz(\C(L))=Sz(\C\big(\big[1, \omega^{\omega^{\alpha}}\big]\big))=\Gamma(ht\big(\big[1, \omega^{\omega^{\alpha}}\big]\big))=\Gamma(\omega^{\alpha}+1)=\omega^{\alpha+1}.\] 
Hence, if
\[\Gamma(ht(K))=Sz(\C(K)) \geq Sz(\C(L))=\omega^{\alpha+1},\] 
then $ht(K)>\omega^{\alpha}$, so $K^{(\omega^{\alpha})}\neq\emptyset$. This, by Proposition \ref{Copies} (for $U=K$), implies that $\C(K)$ contains a PNPP isometric copy of $\C\big(\big[1, \omega^{\omega^{\alpha}}\big]\big)$, and hence an isomorphic copy of $\C(L)$.

If, on the other hand, $L$ is uncountable, then $L$ is not scattered and by the classical Miljutin theorem (\cite[Theorem 2.1]{RosenthalC(K)}), $\C(L)$ is isomorphic to $\C([0, 1])$. Thus, $\infty=Sz(\C(L)) \leq Sz(\C(K))$, so $K$ is not scattered as well, which implies that $K$ maps continuously onto $[0,1]$, and hence that $\C(K)$ contains an isomorphic copy of $\C([0, 1]) \simeq \C(L)$.%, see \cite{Pelczynski_Semadeni_scattered}.

(iv)$\Rightarrow$(i) follows from the main result of Rondo\v{s} and Somaglia \cite{rondos-somaglia}.

(v)$\Rightarrow$(i) It was proved by Causey \cite[Theorem 1.4]{Causey_szlenk_hulls} that for each Banach space $E$ it holds $Sz(\C(K, E))=\max \{Sz(\C(K)), Sz(E)\}$. Thus, we have
\[Sz(\C(L)) \leq Sz(\C(K, E))=\max \{Sz(\C(K)), Sz(E)\}=Sz(\C(K)).\]
\end{proof}

\begin{remark}
    By the discussion preceding Theorem \ref{main-isomorphic} we get that, for a compact metric space $L$ and a compact space $K$, $\C(L)$ isomorphically embeds into $\C(K)$ if and only if $\Gamma(ht(L))\le\Gamma(ht(K))$.
\end{remark}

The proof of Theorem \ref{main-isomorphic} yields the following corollary. 

\begin{cor}
\label{equvialence for some ordinals}
For a countable ordinal $\alpha$ and a compact space $K$ the following assertions are equivalent:
\begin{enumerate}[(i)]
    \item $\C(K)$ contains an isomorphic copy of $\C\big(\big[1, \omega^{\omega^\alpha}\big]\big)$, 
    \item $\C(K)$ contains an isometric copy of $\C\big(\big[1, \omega^{\omega^\alpha}\big]\big)$,
    \item $\C(K)$ contains a positive isometric copy of $\C\big(\big[1, \omega^{\omega^\alpha}\big]\big)$.
\end{enumerate}
\end{cor}
\begin{proof}
Trivially, (iii)$\Rightarrow$(ii)$\Rightarrow$(i). On the other hand, if $\C(K)$ contains an isomorphic copy of $\C\big(\big[1, \omega^{\omega^\alpha}\big]\big)$, then as in the proof of the implication (i)$\Rightarrow$(ii) of Theorem \ref{main-isomorphic} we deduce that $\C(K)$ contains a positive isometric copy of $\C\big(\big[1, \omega^{\omega^{\alpha}}\big]\big)$.
\end{proof}

\begin{remark}
Equivalence (i)$\Leftrightarrow$(ii) in Corollary \ref{equvialence for some ordinals} also follows from a more general result of Alspach \cite[Theorem 4]{Bourgain1979}.
\end{remark}

\section{Cellularity of derived sets of $K$ in terms of $\C(K)$\label{sec:cell}}

In this section we present results describing the relative cellularity of derived sets of a compact space $K$ in terms of the structure of the space $\C(K)$.

\begin{thm}
\label{Cellularity of derived sets}
For a compact space $K$ and a countable ordinal $\alpha$, the following three cardinal numbers coincide:
\begin{enumerate}
    \item $c(K^{(\alpha)}, K)$,
    \item $\sup  \big\{ \abs{\Gamma} \colon \C(K_{\Gamma, [1, \omega^{\alpha}]}) \text{ embeds positively isometrically into } \C(K)\big\}$,
    \item $\sup  \big\{ \abs{\Gamma}\colon  \C(K_{\Gamma, [1, \omega^{\alpha}]}) \text{ embeds isometrically into } \C(K)\big\}$.
\end{enumerate}
If, moreover, $K$ is zero-dimensional, then the above three cardinal numbers coincide also with the following two:
\begin{enumerate}\setcounter{enumi}{3}
    \item $\sup  \big\{ \abs{\Gamma} \colon  K \text{ maps continuously onto } K_{\Gamma, [1, \omega^{\alpha}]}\big\}$,
    \item $\sup  \big\{ \abs{\Gamma} \colon \C(K_{\Gamma, [1, \omega^{\alpha}]}) \text{ embeds lattice and algebraically isometrically into } \C(K)\big\}$.
    \end{enumerate}
\end{thm}

\begin{proof}
For the proof of inequality $(1) \leq (2)$, let $\{U_{\gamma}\colon \gamma \in \Gamma \}$ be a nonempty cellular family in $K$ such that $U_{\gamma}\cap K^{(\alpha)}\neq\emptyset$ for each $\gamma \in \Gamma$. By Proposition \ref{Copies}, for each $\gamma \in \Gamma$ there exists a PNPP isometric embedding $T_{\gamma}\colon\C([1, \omega^{\alpha}])\to\C(K)$ which is supported by the set $U_{\gamma}$. Hence, $\C(K_{\Gamma, [1, \omega^{\alpha}]})$ embeds positively isometrically into $\C(K)$---trivially if $\Gamma$ is finite, and by Proposition \ref{copies of C(K_Gamma_L)} (with $h$ being the constant-$1$ function $\mathbf{1}$ on $K$) if $\Gamma$ is infinite.

Inequality $(2) \leq (3)$ is immediate. 

Inequality $(3) \leq (1)$ follows from Proposition \ref{isometry_and_cellularity}, as it is easy to check that for each countable ordinal $\alpha$ we have 
\[c\Big(\big(K_{\Gamma, [1, \omega^{\alpha}]}\big)^{(\alpha)},\ K_{\Gamma, [1, \omega^{\alpha}]}\Big)=\abs{\Gamma},\]
(recall here that $([1, \omega^{\alpha}])^{(\alpha)}=\{\omega^{\alpha}\}$).

\medskip

For the rest of the proof we assume that $K$ is zero-dimensional. 

Inequalities $(4) \leq (5)$ and $(5) \leq (2)$ are trivial. 

For the proof of inequality $(1) \leq (4)$, let $\{U_{\gamma} \colon \gamma \in \Gamma \}$ be a nonempty family of pairwise disjoint clopen subsets of $K$ such that $U_{\gamma}\cap K^{(\alpha)}\neq\emptyset$ for each $\gamma \in \Gamma$. %We can obviously require that $\bigcup_{\gamma \in \Gamma} U_{\gamma} \neq K$ (otherwise, $\Gamma$ is finite).
By Proposition \ref{ContinuousMaps}, for each $\gamma\in\Gamma$, there exists a continuous surjection $\rho_{\gamma} \colon U_{\gamma} \rightarrow [1, \omega^{\alpha}]$. Then, $K$ maps continuously onto $K_{\Gamma, [1, \omega^{\alpha}]}$---again, trivially if $\Gamma$ is finite, and by Lemma \ref{Mapping onto K_Gamma_L} if $\Gamma$ is infinite.
\end{proof}

For our last main result, describing the relative cellularity of the perfect kernel of a compact space $K$, we will need the following lemma, which follows easily from standard facts.

\begin{lemma}
\label{Tietze}
Let $K$ be a compact space and $U$ be an open subset of $K$ such that $U \cap K^{(\infty)} \neq \emptyset$. Then:
\begin{enumerate}[(i)]
    \item there is a PNPP isometric embedding $T \colon \C([0, 1]) \rightarrow \C(K)$ which is supported by $U$,
    \item if $K$ is zero-dimensional and $U$ is clopen, then $U$ maps continuously onto the Cantor space $\Delta$. 
\end{enumerate}
\end{lemma}

\begin{proof}
 We find open sets $V_1, V_2$ of $K$ such that $\overline{V_1} \subseteq V_2 \subseteq \overline{V_2} \subseteq U$ and $V_1 \cap K^{(\infty)} \neq \emptyset$. Then, $(\overline{V_1})^{(\infty)}\neq\emptyset$, since for each ordinal $\alpha$ we have $V_1 \cap K^{(\alpha)} \subseteq V_1^{(\alpha)} \subseteq \overline{V_1}^{(\alpha)}$ and hence 
$V_1 \cap K^{(\infty)} \subseteq \overline{V_1}^{(\infty)}$. Consequently, $\overline{V_1}$ is a nonscattered compact space and therefore there exists a continuous surjection $\Tilde{\rho} \colon \overline{V_1} \rightarrow [0, 1]$. We extend $\Tilde{\rho}$ to a continuous surjection $\rho:K \rightarrow [0, 1]$ using the Tietze extension theorem. We further find a function $g \in \C(K, [0, 1])$ which is equal to constant $1$ on the set $\overline{V_1}$, and which vanishes outside of the set $V_2$. Then, we define the mapping $T \colon \C([0, 1]) \rightarrow \C(K)$, for every $f \in \C([0, 1])$ and $x\in K$, by
\[T(f)(x)=f(\rho(x))\cdot g(x).\]
Clearly, $T$ is linear and $\spt T(f) \subseteq \overline{V_2} \subseteq U$ for each $f \in \C([0, 1])$, so $T$ is supported by $U$. 

Fix $f\in\C([0,1])$. Since $0 \leq g \leq 1$, it is clear that $\norm{T(f)} \leq \norm{f}$. On the other hand, 
\[\norm{T(f)} \geq \sup_{x \in \overline{V_1}} \abs{f(\rho(x))\cdot g(x)}=\sup_{x \in \overline{V_1}} \abs{f(\Tilde{\rho}(x))}=\norm{f}.\]
Thus, $T$ is an isometric embedding. %, which finishes the proof of (i).
We also have
\[\norm{T(f)^+}=\sup_{x\in K}\max\big\{f(\rho(x))\cdot g(x), 0\big\}=\sup_{x\in K}\max\big\{f(\rho(x)), 0\big\}\cdot g(x)=\]
\[=\sup_{x\in V_2}\max\big\{f(\rho(x)), 0\big\}\cdot g(x).\]
Hence, 
\[\norm{T(f)^+}\ge\sup_{x\in\overline{V_1}}\max\big\{f(\rho(x)), 0\big\}\cdot g(x)=\sup_{x\in\overline{V_1}}\max\big\{f(\rho(x)), 0\big\}=\]
\[=\sup_{t\in[0,1]}\max\{f(t), 0\}=\norm{f^+},\]
and
\[\norm{T(f)^+}\le\sup_{x\in V_2}\max\big\{f(\rho(x)), 0\big\}=\sup_{t\in[0,1]}\max\{f(t), 0\}=\norm{f^+},\]
so $\norm{T(f)^+}=\norm{f^+}$, which shows that $T$ is PNPP, and thus the proof of (i) is finished.

\medskip

If $K$ is zero-dimensional and $U$ is clopen, then $U$ is also a zero-dimensional nonscattered compact space, and hence $U$ maps continuously onto the Cantor space $\Delta$, so (ii) holds, too.
\end{proof}

\begin{thm}
\label{Cellularity of perfect kernel}
For a compact space $K$, the following three cardinal numbers coincide:
\begin{enumerate}[(1)]
    \item $c(K^{(\infty)}, K)$,
    \item $\sup  \big\{ \abs{\Gamma} \colon \C(K_{\Gamma, [0, 1]}) \text{ embeds positively isometrically into } \C(K)\big\}$,
    \item $\sup  \big\{ \abs{\Gamma} \colon \C(K_{\Gamma, [0, 1]}) \text{ embeds isometrically into } \C(K)\big\}$.
\end{enumerate}
If, moreover, $K$ is zero-dimensional, then the above three cardinal numbers coincide also with the following two:
\begin{enumerate}\setcounter{enumi}{3}
    \item $\sup  \big\{ \abs{\Gamma} \colon  K \text{ maps continuously onto } K_{\Gamma, \Delta}\big\}$,
    \item $\sup  \big\{ \abs{\Gamma} \colon \C(K_{\Gamma, \Delta}) \text{ embeds lattice and algebraically-isometrically into } \C(K)\big\}$.
\end{enumerate}
\end{thm}

\begin{proof}
The proof is naturally similar to the one of Theorem \ref{Cellularity of derived sets}. We may assume that $K^{(\infty)}\neq\emptyset$, otherwise the result holds by the Pe\l czy\'{n}ski--Semadeni theorem.

For the proof of inequality $(1) \leq (2)$, let $\{U_{\gamma}\colon \gamma \in \Gamma \}$ be an infinite cellular family in $K$ such that $U_{\gamma}\cap K^{(\infty)}\neq\emptyset$ for each $\gamma \in \Gamma$. By Lemma \ref{Tietze}.(i), for each $\gamma \in \Gamma$ there exists a PNPP isometric embedding $T_{\gamma}\colon\C([0, 1])\to\C(K)$ which is supported by the set $U_{\gamma}$. Thus, by Proposition \ref{copies of C(K_Gamma_L)}, $\C(K_{\Gamma, [0, 1]]})$ embeds positively isometrically into $\C(K)$.

Inequality $(2) \leq (3)$ is trivial. 

Inequality $(3) \leq (1)$ holds due to Proposition \ref{isometry_and_cellularity}, since it is clear that the space $K_{\Gamma, [0, 1]}$ is perfect and it is easily seen that its cellularity is equal to $\Gamma$. 

\medskip

For the rest of the proof assume that $K$ is zero-dimensional. 

Inequality $(4) \leq (5)$ is immediate. 

For the proof of inequality $(5) \leq (1)$, follow the argument for $(3) \leq (1)$ with $K_{\Gamma, \Delta}$ in place of $K_{\Gamma, [0, 1]}$.

To see that $(1) \leq (4)$ holds, let $\{U_{\gamma}\colon \gamma \in \Gamma \}$ be an infinite family of pairwise disjoint clopen subsets of $K$ such that $U_{\gamma}\cap K^{(\infty)}\neq\emptyset$ for each $\gamma \in \Gamma$. %, and such that $\bigcup_{\gamma \in \Gamma} U_{\gamma} \neq K$.
By Lemma \ref{Tietze}.(ii), for each $\gamma\in\Gamma$, there exists a continuous surjection $\rho_{\gamma} \colon U_{\gamma} \rightarrow \Delta$. Thus, $K$ maps continuously onto $K_{\Gamma, \Delta}$ by Lemma \ref{Mapping onto K_Gamma_L}.
\end{proof}

Let us finish the paper with the following remark concerning zero-dimensional compact spaces and their Boolean algebras of clopen sets. 

\begin{remark}
If $K$ and $L$ are zero-dimensional compact spaces, then, by virtue of the so-called Stone duality (see \cite[Chapter 3]{Koppelberg_v1}), $K$ continuously maps onto $L$ if and only if the Boolean algebra $Clopen(K)$ of all clopen subsets of $K$ contains a subalgebra isomorphic to the corresponding Boolean algebra $Clopen(L)$ of $L$. Therefore, if $L$ is also zero-dimensional, statement (ix) of Theorem \ref{main-isometric-zero} can be expressed in the following equivalent way:
    
\medskip
\textit{(ix') $Clopen(L)$ embeds homomorphically into $Clopen(K)$.}\\

%Also, f
For an infinite collection $\{L_\gamma\}_{\gamma\in\Gamma}$ of zero-dimensional compact spaces, the algebra $Clop(K_{\{L_\gamma\}_{\gamma\in\Gamma}})$ is isomorphic to the so-called weak product $\prod_{\gamma\in\Gamma}^w Clop(L_\gamma)$ of the collection $\{Clop(L_\gamma)\}_{\gamma\in\Gamma}$ of the Boolean algebras, see \cite[Proposition 8.10]{Koppelberg_v1}. %Finally,
Also, for the Cantor space $\Delta$, the algebra $Clop(\Delta)$ is isomorphic to the countable free Boolean algebra $Fr(\omega)$, see \cite[Corollary 9.7.(a)]{Koppelberg_v1}. 

Consequently, for a zero-dimensional compact space $K$, items (4) of both Theorems \ref{Cellularity of derived sets} and \ref{Cellularity of perfect kernel} can be expressed in the equivalent Boolean-theoretic way, yielding the following additional equalities (4'):
\textit{\[c(K^{(\alpha)}, K)=\sup  \Big\{ \abs{\Gamma} \colon \prod_{\gamma\in\Gamma}^w Clop([1, \omega^{\alpha}]) \text{ embeds homomorphically into }Clop(K)\Big\}\]}
and
\textit{\[c(K^{(\infty)}, K)=\sup  \Big\{ \abs{\Gamma} \colon \prod_{\gamma\in\Gamma}^w Fr(\omega)  \text{ embeds homomorphically into }Clop(K)\Big\}.\]}
\end{remark}

\bibliography{bibliography}\bibliographystyle{siam}
\end{document}